\documentclass[final,onefignum]{siamltex1213}
\usepackage{float}
\usepackage[english]{babel}
\usepackage{dsfont}
\usepackage{amssymb}
\usepackage{amsmath}
\usepackage{graphicx}
\usepackage{pict2e}
\usepackage{tikz}
\usetikzlibrary{arrows,positioning}
\tikzset{
    >=stealth,
    font=\scriptsize,
    possible world/.style={circle,draw,thick,align=center},
    real world/.style={double,circle,draw,thick,align=center},
    minimum size=40pt
}
\tikzstyle{vertex}=[circle, draw, inner sep=0pt, minimum size=6pt]
\newcommand{\vertex}{\node[vertex]}

\newtheorem{example}{Example}[section]
\newtheorem{remark}{Remark}[section]

\renewcommand{\H}{\mathcal{H}}

\newcommand{\perm}{\mathrm{perm}}
\newcommand{\id}{\mathsf{I}}
\newcommand{\colr}{\color{red}}

\title{Algebraic combinatorics on trace monoids: extending number theory to walks on graphs}

\author{P.-L. Giscard\footnotemark[1]
\and P. Rochet\footnotemark[2]}

\begin{document}
\maketitle

\renewcommand{\thefootnote}{\fnsymbol{footnote}}

\footnotetext[1]{University of York, Department of Computer Sciences. Email: pierre-louis.giscard@york.ac.uk}
\footnotetext[2]{Universit\'e de Nantes, Laboratoire de Math\'ematiques Jean Leray. }

\renewcommand{\thefootnote}{\arabic{footnote}}
 
 \pagestyle{myheadings}
\thispagestyle{plain}
\markboth{P.-L. GISCARD AND P. ROCHET}{EXTENDING NUMBER THEORY TO WALKS ON GRAPHS}

\begin{abstract} 
Partially commutative monoids provide a powerful tool to study graphs, viewing walks as words whose letters, the edges of the graph, obey a specific commutation rule. A particular class of traces emerges from this framework, the hikes, whose alphabet is the set of simple cycles on the graph. We show that hikes characterize undirected graphs uniquely, up to isomorphism, and satisfy remarkable algebraic properties such as the existence and uniqueness of a prime factorization. Because of this, the set of hikes partially ordered by divisibility hosts a plethora of relations in direct correspondence with those found in number theory. Some applications of these results are presented, including a permanantal extension to MacMahon's master theorem and a derivation of the Ihara zeta function.
\end{abstract}



\vspace{6mm}
\section{Introduction} Several school of thoughts have emerged from the literature in graph theory, concerned with studying walks (also called paths) on directed graphs as algebraic objects. Among the numerous structures proposed over the years are those based on walk concatenation such as the path-algebra \cite{Brion2008} and later, nesting \cite{giscard2012continued} or the cycle space \cite{diestel2000graph}. A promising approach using the theory of partially commutative monoids (also called trace monoids), consists in viewing the arcs (i.e. the directed edges) of a graph as letters forming an alphabet and walks as words on this alphabet. A crucial idea in this approach, proposed by \cite{cartier1969}, is to define a specific commutation rule on the alphabet: two arcs commute if and only if they initiate from different vertices. This construction yields a semi-commutative monoid which allows for a great flexibility in the walk structure while preserving the ability to distinguish between different walks composed of the same arcs. A remarkable consequence of this construction is the existence of a stable subset of traces, formed by collections of cycles: the hikes. More specifically, hikes constitute a simplified trace monoid that carries most of the information pertaining to the graph structure and, in the case of undirected graphs, all the information. In this trace monoid, the simple cycles form the alphabet while the independence relation is characterized by vertex-disjointness.\\

Of fundamental importance for the trace monoid of hikes is the hitherto underappreciated prime-property satisfied by the simple cycles. Recall that an element of a monoid is prime if and only if, whenever it is factor of the product of two elements, then it is a factor of at least one of the two. The importance of the prime property lies in that because of it, the partially ordered set formed by the hikes ordered by divisibility is host to a plethora of algebraic relations in direct extension to number theory. 
This includes identities involving many more objects beyond the well-studied M\"{o}bius function \cite{cartier1969,rota1987foundations}, such as the von Mangoldt and Liouville functions.
In this respect hikes are natural objects to consider, as most of their algebraic properties follow from analytical transformations of the weighted adjacency matrix. The study of the algebraic structures associated with hikes is the main subject of the present work. These structures provide an extended semi-commutative framework to number theory from which both well-known and novel relations in general combinatorics are derived as particular consequences.\\

The article is organized as follows. In Section~\ref{sec:2}, we present the theoretical setting required for the construction of hikes as elements of a specific trace monoid. We discuss some immediate consequences, such as the uniqueness of the prime decomposition, the hike analogue of the fundamental theorem of arithmetic. 

Section~\ref{AlgHike} is devoted to the study of algebraic relations between formal series on hikes. In particular, we discuss the hike version of Riemann's zeta function, defined as a formal series of the constant function $1$ over hikes. We then introduce the von Mangoldt function $\Lambda$ over hikes in \S\ref{MangoldtSec} whose formal series stems from a log-derivative formula involving the zeta function, and discuss its connections to closed walks. Further general properties concerning totally additive and totally multiplicative functions over the hikes are derived in \S\ref{TotAddSec} and \S\ref{TotMultSec}. In particular, we show that: i) the convolution of a totally additive function $f$ with the M\"obius function $\mu$ has its support over walks, ii) the inverse through the Dirichlet convolution of a totally multiplicative function $f$ is given by $f^{-1} = f \mu$. These results are illustrated on typical additive functions such as the length $\ell$ or the prime factor counting function $\Omega$, recovering generalisations of classical results of number theory. Another illustration concerns the walk Liouville function which arises from a permanent expression of the labelled adjacency matrix. We establish the relation between these results and their number-theoretic counterparts in \S\ref{NumberTheory}, by showing that there exists a class of infinite graphs on which the hike poset is isomorphic to the poset of integers ordered by divisibility. 

In Section~\ref{IharaZeta} we elucidate the connection between the hike zeta function and the Ihara zeta function $\zeta_I$ of the graph $G$. This connection suggests that hikes hold more information than $\zeta_I$, something we confirm in \S\ref{PGdetermines} by showing that the hike poset determines undirected graphs uniquely, up to isomorphism. Future perspectives and possible extensions of our work are discussed in the conclusion.

\section{General setting}\label{sec:2} Let $G=(V,E)$ be a directed graph with finite vertex set $V = \{v_1,\hdots,v_N \}$ and edge set $E \subseteq V^2$, which may contain loops. An element of $E$ is called a directed edge or an arc. Let $\mathsf{W} = (w_{ij})_{i,j\,=\,1,\hdots,N}$ represent the weighted adjacency matrix of the graph, built by attributing a formal variable $w_{ij}$ to every pair $(v_i,v_j ) \in V^2$ and setting $w_{ij}=0$ whenever there is no arc from $v_i$ to $v_j$ in $G$. In the sequel, we identify each arc of $G$ to the corresponding non-zero variable $w_{ij}$.
\\

A walk, or path, of length $\ell$ from $v_i$ to $v_j$ on $G$ is a contiguous sequence of $\ell$ arcs starting from $v_i$ end ending to $v_j$, e.g. $p = w_{i i_1} w_{i_1 i_2} \cdots w_{i_{\ell-1} j}$ (a sequence of arcs is said to be contiguous if each arc but the first one starts where the previous ended). The walk $p$ is \textit{open} if $i \neq j$ and \textit{closed} (a cycle) otherwise. A backtrack is a cycle of length two, e.g. $w_{ij} w_{ji}$ while a self-loop $w_{ii}$ is considered a cycle of length one.

A walk $p$ is \textit{self-avoiding}, or \textit{simple}, if it does not cross the same vertex twice, that is, if the indices $i,i_1,\hdots,i_{\ell-1},j$ are mutually different (with the possible exception $i=j$ if $p$ is closed).

\subsection{Partially commutative structure on the arcs} We endow the arcs $w_{ij}$ with a partially commutative structure which allows the permutation of two arcs only if they start from different vertices. This particular setting was initially considered in  \cite{cartier1969}. In this section, we discuss the motivations and implications of this structure to study walks and cycles on a graph.\\

\noindent \textbf{Commutation rule:} Two different arcs $w_{ij}$ and $w_{i'j'}$ commute if, and only if, $i \neq i'$.\\

\noindent The finite sequences of arcs form a free partially commutative monoid $\mathcal M$, also called \textit{trace monoid}, with alphabet $\Sigma_{\mathcal M} := \{ w_{ij}: w_{ij} \neq 0 \}$ and independence relation
$$ I_{\mathcal M} = \{ (w_{ij},w_{kl}) : i \neq k \}.   $$
A trace $t$, i.e. an element of $\mathcal M$, can be viewed as an equivalence class in the free monoid generated by the arcs. The equivalence relation is then defined as follows: two sequences of arcs $s, s'$ are equivalent if $s$ can be obtained from $s'$ upon permuting arcs with different starting points. Different elements of an equivalence class $t \in \mathcal M$ will be referred to as \textit{representations} of a trace.

In this setting, a walk (a sequence of contiguous arcs) may have non-contiguous representations. For instance, the walk $w_{12} w_{23}$ from $v_1$ to $v_3$ can be rewritten as $ w_{23} w_{12}$ since $w_{23}$ and $w_{12}$ start from different vertices. In fact, an open walk always has a unique contiguous representation, as any allowed permutations of arcs would break the contiguity. 
Surprisingly, the uniqueness of the contiguous representation no longer holds for closed walks. This consequence is an important feature of the partially commutative structure on the arcs: two closed walks starting from different vertices define the same object if they can be obtained from one another by permuting arcs with different starting points.\\

To illustrate this statement, consider the example pictured in Figure~\ref{fig:ex1}. 
The only closed walk starting from $v_1$ that covers every arc exactly once is $c_1:= w_{12} w_{23} w_{34}  w_{45} w_{53} w_{31} $. Since the only non-commuting arcs are $w_{31}$ and $w_{34}$, the cycle can be rewritten as starting from $v_3$ by $c_1 = w_{34} w_{45} w_{53} w_{31} w_{12} w_{23}$.
On the other hand, there are two closed walks starting from $v_3$ covering every arc once, one is $c_1$ and the other is $c_2:= w_{31} w_{12}w_{23}  w_{34} w_{45} w_{53}$. One cannot go from $c_1$ to $c_2$ without permuting $w_{31}$ and $w_{34}$, thus $c_1$ and $c_2$ are different elements in $\mathcal M$. Here, $w_{12} w_{23} w_{31}$ and $ w_{31} w_{12} w_{23}$ are equivalent representations of the same cycle since the permutations of the arcs to go from one to the other are allowed. More generally, the starting vertex of a simple cycle never influences its value.
\begin{figure}[h]
\begin{center}
\includegraphics[width=0.4\textwidth]{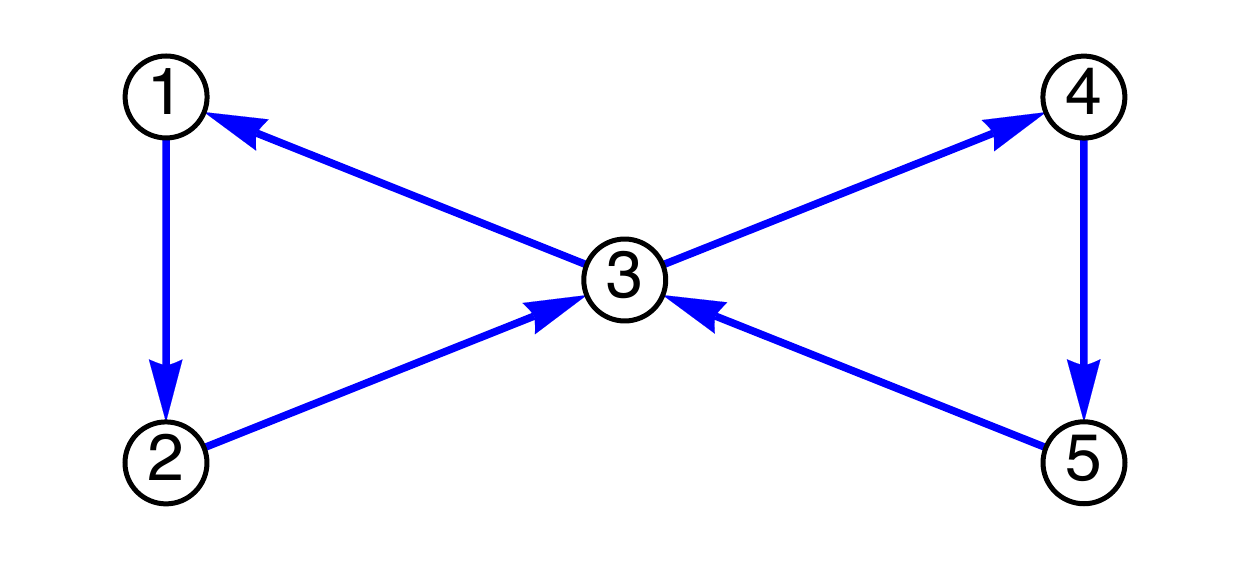}
\vspace{-2mm}
\caption{\footnotesize{The closed walks $c_1 = w_{34} w_{45} w_{53}  w_{31} w_{12} w_{23}$ and $c_2 =  w_{31}w_{12} w_{23} w_{34} w_{45} w_{53}$ are different although composed of the same arcs. Both are achievable starting from $v_3$ but only $c_1$ is achievable from $v_1$.}}
\label{fig:ex1}
\end{center}
\end{figure}

\subsection{Multiplication and factorization of hikes} A closed walk can be characterized as a contiguous sequence of arcs comprising the same number of ingoing and outgoing arcs for each vertex. A hike is obtained upon relaxing the contiguity condition:
\vspace{2mm} 

\begin{definition} A hike is a trace $h=w_{i_1 j_1} \cdots w_{i_\ell j_\ell} \in \mathcal M$ for which the numbers of outgoing and ingoing arcs for each vertex are equal. Formally, the indices $i_k,j_k, k =1,..., \ell$ of the arcs of $h$ satisfy,
\begin{equation}\label{altdef} \forall i=1,\hdots,N \ , \ \# \{ k: i_k = i \} = \# \{k:  j_k = i \}, \end{equation}
where $\#$ stands for the cardinality.
\end{definition}
\vspace{2mm}

\begin{remark} 
Hikes have been studied in various forms in the literature. They originated under the name "circuits" in the seminal paper \cite{cartier1969}, written in french. We here use a different term to avoid confusion, as circuit may refer to other common objects in graph theory. Hikes can also be defined as the particular heaps of pieces formed from the simple cycles of the graph, see \cite{viennot1989heaps}. More recently, their role in gauge theory has been investigated in \cite{mattioli2014quivers}.  \\
\end{remark}

We denote by $\mathcal H$ the set of hikes, which is a subset of $\mathcal M$. By convention, the trivial walk $1$ viewed as the empty sequence is considered to be a hike. We emphasize that, since hikes are elements of $\mathcal M$, they obey the partially commutative structure on the arcs: two hikes $h$ and $h'$ are equal if, and only if, $h'$ can be obtained from $h$ by permuting arcs in $h$ with different starting point. In particular, while every closed walk is a hike, a hike is a closed walk only if it has a contiguous representation. Moreover, we say that a hike is self-avoiding if all its arcs  are different and commute.\\

The multiplication of two hikes $h,h'$, simply defined as the concatenation, yields a hike and shall be denoted by $h.h'$ or simply $hh'$ in the sequel. We define hike division as the reverse operation: $d \in \mathcal H$ divides $h \in \mathcal H$, which we write $d | h$, if there exists $h' \in \mathcal H$ such that $h = d . h'$. We shall use the standard division notation
$$ h = d . h' \ \Longleftrightarrow \ h' = \frac h d.  $$
Here the choice of left-division, rather than right-division, is only a matter of convention. Remark that because the multiplication of hikes is not commutative, $d | h$ does not necessarily imply that $h/d$ divides $h$. \\

\begin{remark} While it may be tempting to identify a hike with an Eulerian (sub-)digraph, the two are nonetheless different objects. An obvious reason is that a hike may contain several appearances of the same arc. Moreover, the non-commutative structure on the arcs induces an order on the oriented cycles it contains. This implies that an Eulerian multi-digraph can not in general be associated with a unique hike. \\
\end{remark}

Every non-trivial hike $h$ has a representation as a product of simple cycles $h = c_1\cdots c_k$ which can be obtained by considering the simple cycles in the order they are formed in the arc sequence. This construction is somewhat similar to Lawler's loop-erasing procedure \cite{lawler1987loop} which divides a closed walk into a finite sequence of simple cycles. By considering hikes, we argue that this decomposition remains natural when relaxing the contiguity condition. 

A main concern of this paper is to treat the decomposition of a hike into simple cycles as a \emph{prime decomposition}, seeing the simple cycles as prime factors. Rigorously, an element $p$ of a monoid is prime if and only if, whenever $p$ is a factor of $a.b$, then $p$ is a factor of $a$ or $b$. Thus, in this context, the primes are indeed given by the  simple cycles. We emphasize that, because of the lack of commutativity, the prime factors of $h$, i.e. the elements of the prime decomposition, are different from its prime divisors.\\

An important consequence of the commutation rule on the arcs is that the prime decomposition $h = c_1\cdots c_k$ is unique up to permutations of consecutive vertex-disjoint simple cycles. Indeed, switching two different consecutive cycles in the prime decomposition $h = c_1 \cdots c_k$ violates the commutation rule as soon as $V(c_i) \cap V(c_{i+1}) \neq \emptyset$. This property highlights that $\mathcal H$ forms a sub-monoid of $\mathcal M$, whose alphabet is the set of prime hikes (the simple cycles) $\Sigma_{\mathcal H} := \{c_1,\hdots,c_k \}$ and with independence relation defined by
\begin{equation}\label{dephike} I_{\mathcal H} = \big\{ (c_i, c_j) : V(c_i) \cap V(c_j) = \emptyset \big\}. \end{equation}

A geometric interpretation of the prime decomposition can be found in Viennot's theory of \textit{heaps of pieces} \cite{viennot1989heaps}.  In this case, the simple cycles are pieces "piled up" in such a way that two simple cycles can only be put to the same level if they share no vertex in common. In fact, heaps of pieces provide a geometric construction for the Cartier-Foata clique decomposition of a trace.\\

\begin{definition} The independence graph of $\mathcal H$ is the undirected graph whose vertices are the simple cycles of $G$ and with an edge between two simple cycles  $c,c' \in \Sigma_\mathcal H$ if they share no vertex in common.\\
\end{definition}

\begin{remark}
In the general case of a trace monoid $\mathcal T$ with alphabet $\Sigma_\mathcal T$ and independence relation $\mathcal I_\mathcal T$, some authors simply define the independence graph of $\mathcal T$ as the directed graph $(\Sigma_\mathcal T, \mathcal I_\mathcal T)$, see e.g. \cite{krob2003computing}. Nevertheless, since the independence relation $\mathcal I_T$ is symmetric, considering the undirected version induces no loss of information.\\
\end{remark}

A clique of the independence graph of $\mathcal H$ can be identified with a product of pair-wise non-intersecting simple cycles, that is, a self-avoiding hike. The Cartier-Foata clique decomposition of a hike $h$ can then be built as follows.

\begin{enumerate}
\item[i)] The maximal self-avoiding divisor of a hike $h$ (i.e. the product $s(h)$ of its prime divisors), is the first clique in the Cartier-Foata decomposition of $h$.
	\item[ii)] If $h$ is self-avoiding, then $s(h)=h$ and $h$ is its own Cartier-Foata decomposition. All simple cycles composing $h$ are vertex disjoint so that they can be put to the same level, forming a heap of height $1$.
	\item[iii)] Otherwise, consider a collection of self-avoiding hikes $s_k$, initiated by $s_1 = s(h)$, and setting
$$s_{k+1} = s \Big( \frac{h}{s_1 \cdots s_{k} } \Big) $$
until all simple cycles of $h$ are made part of a clique $s_k$. Each clique of the Cartier-Foata decomposition defines a layer in the heap of pieces. 
\end{enumerate}
We refer to \cite{krattenthaler2006theory} for a detailed explanation of the representation of heaps of pieces and its equivalence with Cartier-Foata monoids. \\

In the sequel, $\ell(h)$ represents the length of a closed hike $h$ while the number of elements in its prime-decomposition (counted with multiplicity) is denoted by $\Omega(h)$. By convention, the trivial hike 1 is not prime and thus $\Omega(1)=0$.

\subsection{Hikes incidence algebra} The hikes ordered by division form a locally finite partially ordered set, or \textit{poset}, which we denote $P_G$. The reduced incidence algebra on this poset is the set $\mathcal F$ of real-valued functions on $\mathcal H$ endowed with the Dirichlet convolution
$$ f * g (h) := \sum_{d | h} f(d) g \Big( \frac h d \Big) \ , \ h \in \mathcal H.  $$
Here, the sum is taken over all left-divisors $d$ of $h$, including $h$ itself and the trivial hike $1$. One verifies easily that the Dirichlet convolution is associative and distributive over addition. However, it is not commutative since $d$ can divide $h$ without it being the case for $h /d$.\\

\begin{proposition} The reduced incidence algebra $(\mathcal F,*)$ is isomorphic to the algebra of formal series
$$\sum_{h \in \mathcal H} f(h) h  \ , \ f \in \mathcal F $$ 
endowed with hike multiplication.\\
\end{proposition}

The formal series over hikes may be understood as abstract representations of the elements of $\mathcal F$, in the sense that two functions $f,g \in \mathcal F $ are equal if, and only if their formal series are identical.  The hike multiplication naturally extends over formal series by setting for $f,g \in \mathcal F$,
$$ \bigg( \sum_{h \in \mathcal H} f(h) h  \bigg) . \bigg( \sum_{h' \in \mathcal H} g(h') h'  \bigg) : = \sum_{h ,h' \in \mathcal H} f(h) g(h') h h' . $$
By regrouping on the right-hand side the couples $(h,h')$ for which the product $h.h'$ are equal, one recovers the formal series of the convolution $f*g$,
$$ \sum_{h ,h' \in \mathcal H} f(h) g(h') h h' = \sum_{h \in \mathcal H} \bigg( \sum_{d | h} f(d) g \Big( \frac hd \Big) \bigg) h = \sum_{h \in \mathcal H}  f*g(h) h ,$$
thus the proposition.\\

\noindent Important functions of the reduced incidence algebra include the identity $\delta(.)$ equal to one for $h=1$ and zero otherwise, the constant function $1(h)=1 \ , \ \forall h \in \mathcal H$ or the M\"{o}bius function, the inverse of $1$ through the Dirichlet convolution. We refer to \cite{rota1987foundations} for a more comprehensive study. It is one of the main results of the present work that many more number-theoretic functions beyond $1$ and $\mu$ have generalized analogs in the reduced incidence algebra $(\mathcal{F},\ast)$ and that these analogs satisfy the same relations as their number-theoretic counterparts, see \S\ref{AlgHike}.
\vspace{2mm}

The next theorem gives the expression of the M\"{o}bius function on $\mathcal H$. This result is discussed in Remark 3.6 in \cite{cartier1969}. Nevertheless, we provide an elementary proof for the sake of completeness.
\vspace{2mm}

\begin{proposition}\label{th:mob} The M\"{o}bius function on $\mathcal H$ is given by
\begin{equation}\label{mobius} \mu(h) := \left\{ \begin{array}{cl} 
1 & \text{if $h=1$} \\ 
(-1)^{\Omega(h)} & \text{if $h$ is self-avoiding} \\ 0 & \text{otherwise.}  \end{array} \right.  \end{equation}
\end{proposition}
\begin{proof}The M\"{o}bius function on $\mathcal H$ is the inverse of the constant function $1$, i.e. the unique function $\mu$ such that $\mu(1) = 1$ and
\begin{equation}\label{mobdir} \forall h \neq 1 , \ \mu *1(h)  = \sum_{d | h} \mu(d) 1 \Big( \frac h d \Big) = \sum_{d | h} \mu(d) = 0.\end{equation}
We need to verify that $\mu$ as defined in \eqref{mobius} satisfies this relation. Let $h\neq 1$ and $s(h)$ denote the largest self-avoiding divisor of $h$, i.e. the product of all its prime divisors. Since the self-avoiding divisors of $h$ are the divisors of $s(h)$ and $\mu(d) =0$ whenever $d$ is not self-avoiding, it follows that $\mu*1(h) = \mu*1(s(h))$. So, it suffices to show the result for $h$ self-avoiding. We proceed by induction. If $h=c$ is a simple cycle, then we verify easily the relation
$$ \mu*1(c) = \mu(1) + \mu(c) = 1 - 1 = 0.   $$
Now, let $c_1, \hdots, c_k, c_{k+1}$ be vertex-disjoint cycles and assume that \eqref{mobdir} holds for $h=c_1\hdots c_k$. We have
$$ \sum_{d | h. c_{k+1}} \mu(d) = \sum_{d | h} \mu(d) +  \sum_{d | h} \mu(d.c_{k+1})  = \sum_{d | h} \mu(d) - \sum_{d | h} \mu(d) = 0, $$
ending the proof.\qquad
\end{proof}
\vspace{2mm}

Proposition \ref{th:mob} confirms the characterization of $\H$ as the trace monoid generated by the alphabet of simple cycles $\Sigma_{\H} =\{c_1,\hdots,c_k \} $ with independence relation defined in Equation \eqref{dephike} (see Eq. (56), Chapter 2.5 in \cite{sandor2004handbook}). The formal series associated to the M\"{o}bius function for $\mathcal H$ then appears in the identity 
\begin{equation}\label{mobius_inv} \det(\id - \mathsf{W}) = \sum_{h \in \mathcal H} \mu(h) h,  \end{equation}
where we recall that $\mathsf W$ denotes the labelled adjacency matrix of $G$. A proof of this identity can be found in Theorem~1 of \cite{ponstein1966self} on noting that for self-avoiding hikes, the concatenation of arcs coincides with the ordinary multiplication. Proposition \ref{th:mob} thus provides a determinant formula for the M\"{o}bius function of $\H$ and the formal series associated to $1$ (i.e. the analogue of the zeta function) is obtained via the formal inversion
$$ \det(\mathsf{I} - \mathsf{W})^{-1} = \frac{1}{ \sum_{h \in \mathcal{H}} \mu(h) h} 
= \sum_{h \in \mathcal{H}} h.  $$

\begin{remark}[Coprimality] The M\"{o}bius function is multiplicative on vertex-disjoint hikes,
\begin{equation}\label{multiplicative}  V(h) \cap V(h') = \emptyset  \ \Longrightarrow \ \mu(hh') = \mu(h) \mu(h'). \end{equation}
This identity is reminiscent of the multiplicative property of the number-theoretic M\"{o}bius function $\mu_{\mathbb{N}}$ for which $\mu_{\mathbb{N}}(nm)=\mu_{\mathbb{N}}(n)\mu_{\mathbb{N}}(m)$ whenever $n$ and $m$ are coprime integers. The fact that \eqref{multiplicative} only holds for vertex-disjoint hikes suggests a more general notion of coprimality on $\mathcal H$: two hikes are coprime if they share no vertex in common. In particular, coprime hikes have different prime factors, but contrary to natural integers, this condition is in general not sufficient. The two notions of coprimality coincide on a class of graphs where $\mu_{\mathbb{N}}$ is recovered from $\mu$, see \S\ref{NumberTheory}.\\
\end{remark}

A determinantal expression for the M\"{o}bius function of certain trace monoids $\mathcal{M}$ that is similar to Eq.~(\ref{mobius_inv}) was obtained in \cite{choffrut1999determinants}. 
Nevertheless, these two results have different domains of validity
 and  
arise from different constructions.  In \cite{choffrut1999determinants}, the expression of the Mobius function as $\det(\id - \mathsf{X})$
involves a matrix $\mathsf{X}$ whose entries are polynomials in the letters of the trace monoid. Furthermore, this formula holds if and only if the independence relation admits a transitive orientation (see  \cite{diekert1988transitive}). 

The situation is different for Eq.~(\ref{mobius_inv}) since the weighted adjacency matrix $\mathsf{W}$ involves the arcs of $G$, which are \textit{subdivisions} of the simple cycles of $G$ and thus subdivisions of the letters of $\mathcal{H}$. In this setting, a transitive orientation is not necessary anymore for Eq.~(\ref{mobius_inv}) to hold, since
$\H$ is not necessarily transitively orientable (see Example~\ref{ExNoTransitive} below). This means that the determinant formula Eq.~\eqref{mobius_inv} holds in situations where the result of \cite{choffrut1999determinants} does not apply. On the other hand, the reverse is also true: there exist transitively orientable trace monoids which do not constitute hike trace monoids. For instance, one can show with little work that no digraph has the cycle on six vertices $C_6$ as hike independence graph, yet one can easily construct a transitively orientable trace monoid with independence graph $C_6$. 
Thus, our result and those of \cite{choffrut1999determinants, diekert1988transitive} seem to be complementary. A complete characterization of hike trace monoids is beyond the scope of this work.
\vspace{2mm}

\begin{example}[Hike trace monoid with no transitive orientation and a determinantal M\"{o}bius function]\label{ExNoTransitive}
Let $G$ be the cycle graph on 5 vertices illustrated in Figure \ref{fig:tracemon}. There are seven primes on $G$: $a=w_{13}w_{31}$, $b=w_{24}w_{42}$, $c=w_{35}w_{53}$, $d=w_{14}w_{41}$, $e=w_{25}w_{52}$, $f=w_{13}w_{35}w_{52}w_{24}w_{41}$ and $g=w_{14}w_{42}w_{25}w_{53}w_{31}$. 
\begin{figure}[H]
\vspace{-0mm}
	\centering
		\includegraphics[width=0.25\textwidth]{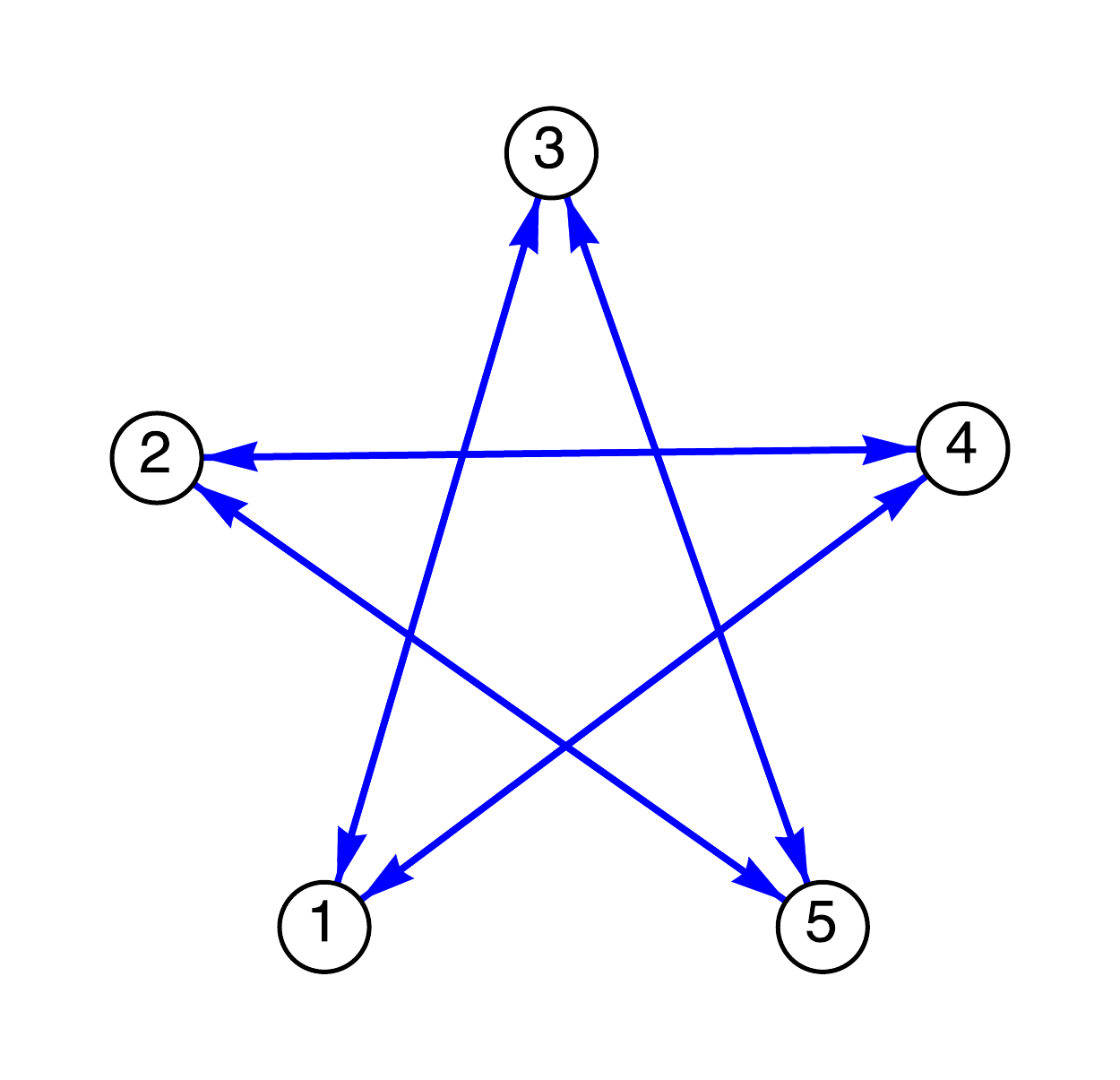} \hspace{2cm} 
		\includegraphics[width=0.25\textwidth]{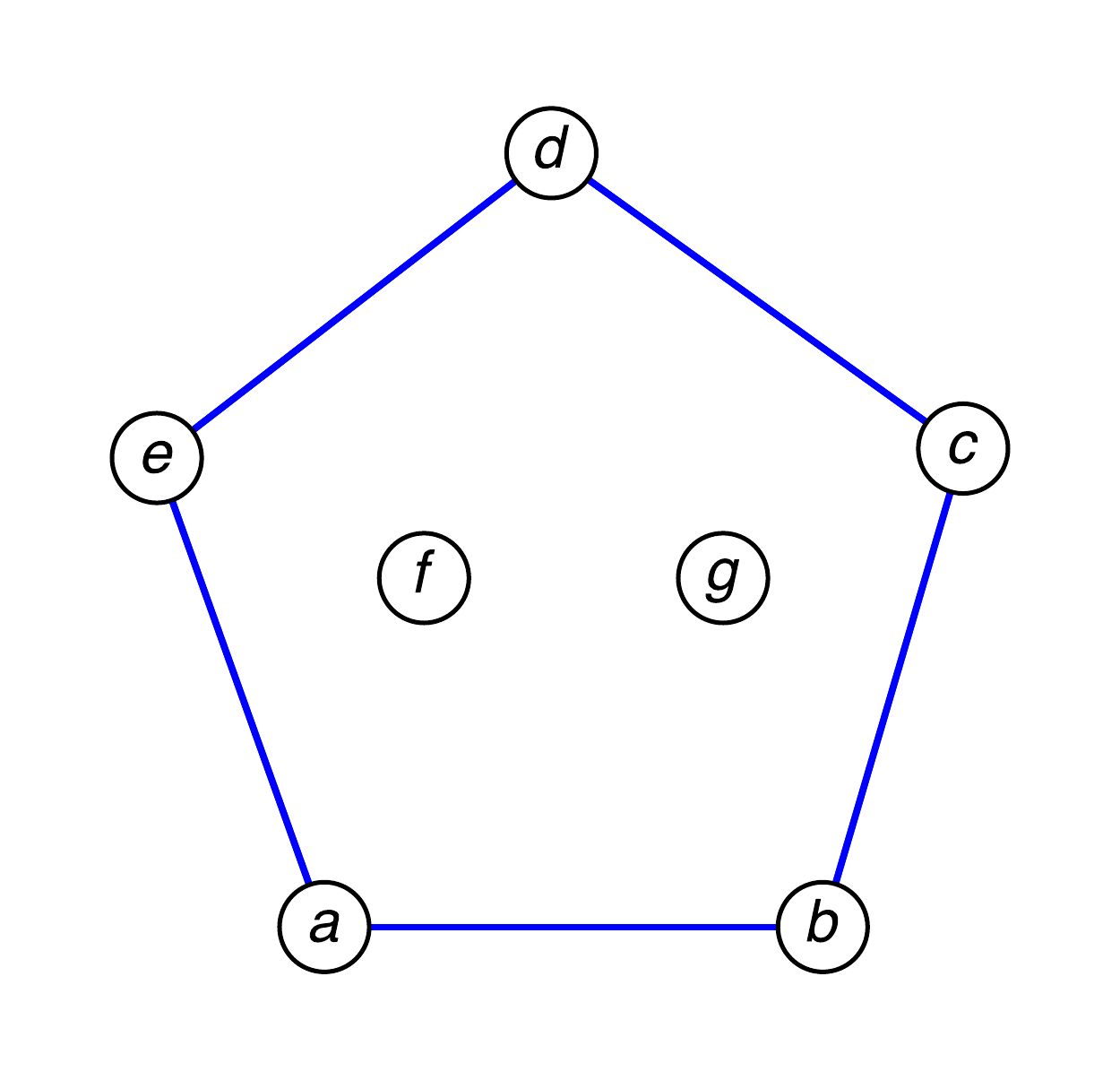}
		\caption{\label{C5graph}The bidirected cycle graph on 5 vertices (left) and its hike independence graph (right).}
	\label{fig:tracemon}
	\vspace{-3mm}
\end{figure}
\noindent Therefore, the hikes on $G$ form the trace monoid $\H$ on seven letters $\Sigma_\mathcal H = \{a,b,c,d,e,f,g\}$ with independence relation
$$ I_\H = \{ (a,b),(b,a),(b,c),(c,b),(c,d),(d,c),(d,e),(e,d),(e,a),(a,e) \}. $$
The commutation graph of $\mathcal H$, presented on Figure \ref{fig:tracemon} (right), is not a comparability graph since it contains a cycle of length $5$ as an induced subgraph (see also Example 11.ii of \cite{diekert1988transitive}). 
Consequently, $I_\H$ does not admit a transitive orientation, yet Eq.~\eqref{mobius_inv} indicates that
\begin{align*}
\sum_{h\in\H} \mu(h) h&=1-a-b-c-d-e-f-g+ac+ad+bd+be+ce,\\
&=\det(\mathsf{I}-\mathsf{W})=\det\left(
\begin{array}{ccccc}
 1 & 0 & -w_{13} & -w_{14} & 0 \\
 0 & 1 & 0 & -w_{24} & -w_{25} \\
 -w_{31} & 0 & 1 & 0 & -w_{35} \\
 -w_{41} & -w_{42} & 0 & 1 & 0 \\
 0 & -w_{52} & -w_{53} & 0 & 1 \\
\end{array}
\right),
\end{align*}
that is, the M\"{o}bius function of $\H$ admits a determinantal form.
\end{example}

\section{Algebraic relations between series on hikes}\label{AlgHike} 
In this section we show that a plethora of number theoretic relations find natural extensions on the trace monoid of hikes. These provide powerful algebraic tools in a novel graph theoretic context and yield further insights into well established results. For example, we find in \S\ref{TotMultSec} that MacMahon's master theorem and the Dirichlet inverse of totally multiplicative functions over the integers both originate from the same general result about series of hikes. Throughout this section, $G$ designates a digraph and $P_G$ is the poset of hikes on $G$ ordered by divisibility.\\

\begin{definition} The formal series associated to a function $f\in\mathcal{F}$ is defined as 
$$\mathcal{S}f(s):= \sum_{h\in\H}e^{-s\ell(h)} f(h) h$$
for $s$ a complex variable. The particular case $f=1$ is the zeta function $\zeta(s) := \mathcal{S}1(s)$. \\
\end{definition}

Recall that because of the lack of commutativity between hikes, Dirichlet convolution typically acts non-commutatively on functions on hikes $g*f\neq f*g$ and thus hike-series also multiply non-commutatively, i.e. $\mathcal{S} f.\mathcal{S} g=\mathcal{S}(f*g)\neq \mathcal{S} g.\mathcal{S} f=\mathcal{S}(g*f)$.
For convenience, we write $\mathcal{S}f /\mathcal{S}g$ for the \textit{right} multiplication with the inverse: $\mathcal{S}f/\mathcal{S}g = \mathcal{S}f.(\mathcal{S}g)^{-1}$.\\

We begin with two simple relations counting the left divisors and left prime divisors of a hike.
\vspace{1mm} 
\begin{proposition}\label{GeneralResults} Let $\tau(h)$ be the number of left divisors of $h\in\H$. Then $$\mathcal{S}\tau(s)=\zeta^2(s).$$
Let $1_{p}$ be the indicator function on primes and $\omega(h)$ the number of prime divisors of $h$. Then
$$
\mathcal{S}\omega(s)=\mathcal{S}1_{p}(s)\,.\, \zeta(s).
$$
 \end{proposition}
\begin{proof}
The results follow immediately from combinatorial arguments on the reduced incidence algebra of $P_G$. First, we have $\tau(h)=\sum_{d|h} 1 = (1*1)(h)$. For the second result, observe that $(1_p*1)(h)=\sum_{d|h} 1_p(d) =\omega(h)$ counts the distinct left prime divisors of $h$.\qquad
\end{proof}
\vspace{2mm}

\begin{remark} The functions $\Omega$ and $\omega$ coincide over self-avoiding hikes, unlike their number theoretic counterparts for which $\Omega(n)=\omega(n)$ if and only if $n$ is a square-free integer. This is due to the stronger characterization of the co-primality for hikes, which requires that the prime factors be not only different but also vertex-disjoint. Indeed, if two different prime factors of a hike $h$ do intersect, then at least one them is not a left-divisor, resulting in $\Omega(h)$ being greater than $\omega(h)$. The parallel with number theory is however accurate if all the simple cycles commute, in which case a prime factors is always a divisor.\\
\end{remark}

While the relations satisfied by the functions $\tau$ and $\omega$ stem from straightforward combinatorial arguments, more advanced algebraic concepts also have natural extensions on the monoid of hikes. 

\subsection{Walk von Mangoldt function}\label{MangoldtSec} We begin with a hike version of the number theoretic von Mangoldt function. \\

\begin{definition}\label{wLambda}
The walk von Mangoldt function $\Lambda:~\H\to\mathbb{N}$ is defined as the number of contiguous representations of a hike that is, $\Lambda(h)$ is the number of possible contiguous rearrangements of the arcs in $h$, obtained without permuting two arcs with the same starting point. \\
\end{definition}

By convention, we set $\Lambda(1)=0$. Remark that by this definition, $\Lambda(h)=0$ whenever $h$ is not a walk. Since different contiguous representations of the same walk start from different vertices, the series associated to $\Lambda$ is obtained by
\begin{equation}\label{Lambda_trace} \mathcal S \Lambda(s) = \sum_{h \in \mathcal H} e^{-s \ell(h)} \Lambda(h) h = \mathrm{Tr} \big(e^{-s} \mathsf W + e^{-2s} \mathsf W^2+... \big) = \mathrm{Tr} \big( (\id - e^{-s}\mathsf W)^{-1}\big) - N.  \end{equation}

The heaps of pieces point of view provides a remarkable characterization of closed walks as heaps of cycles with a unique high-most element (such heaps are called pyramids). Using our terminology, this translates into a (non-trivial) closed walk being a hike with a unique prime right-divisor, whose length is precisely the von Mangoldt function. To go even further, Viennot remarked that each piece in a heap can be associated to the pyramid formed by the pieces below it, including itself, thus revealing the bijection existing between the pieces composing a heap and its sub-heaps containing only one maximal element (this fact is discussed in the proof of Proposition 6 in \cite{viennot1989heaps}). In the context of hikes, this signifies that the prime factors a hike $h$ can be put in bijection with the non-trivial walks dividing it. In particular, summing the von Mangoldt function over all divisors of a hike $h$ reduces to summing the length of its prime factors:
$$ \forall h \in \mathcal H \ , \  \sum_{d | h} \Lambda(d) = \ell(h) \ \Longleftrightarrow \ \Lambda*1 = \ell  \ \Longleftrightarrow \ \Lambda= \ell*\mu.  $$
This powerful observation yields the following result as an immediate consequence.\\

\begin{proposition}\label{LogDerivative} The von Mangoldt function is linked to the zeta function by the following relation
\begin{equation}\label{zprimeoverz}
 \mathcal S \Lambda(s) =\sum_{h \in \mathcal H}e^{-s\ell(h)}\Lambda(h) h = - \frac{\zeta'(s)}{\zeta(s)}
\end{equation}
where $\zeta'(s):=d\zeta(s)/ds = - \sum_{h \in \mathcal H} e^{-s \ell(h)} \ell(h) h$.\\
\end{proposition}

\begin{proof} The equality $\Lambda*1=\ell$ gives
$$  \mathcal S \Lambda(s)  . \zeta(s) = \mathcal S (\Lambda*1) (s) =  \mathcal S \ell(s) = \sum_{h \in \mathcal H} e^{-s \ell(h)} \ell(h) h = - \zeta'(s). $$
The result follows from dividing by $\zeta(s)$. \qquad
\end{proof}\\

While the identity  $\Lambda=\ell*\mu$  stems directly from Viennot's bijection between the prime factors of a hike and the walks dividing it, we provide in the next section an alternate proof that uses only the properties of the M\"obius function.\\

Equation \eqref{zprimeoverz}, which is none other than the series version of the identity $\Lambda = \ell* \mu$, coincides with the number theoretic formula of the von Mangoldt series obtained as the logarithmic derivative of the zeta function. The indefinite integral with respect to $s$ yields the logarithmic identity,
\begin{equation}\label{log_zeta} \log \zeta(s) := \int   \frac{\zeta'(s)}{\zeta(s)} ds = \sum_{h \in \mathcal H}e^{-s\ell(h)}\frac{\Lambda(h)}{\ell(h)} h,  \end{equation}
also reminiscent of its number theoretic counterpart, with the length of a hike playing the role of the logarithm of an integer. Further analogies with number theory are discussed in Section~\ref{NumberTheory}.\\

\begin{example}
\end{example}
To illustrate the relation $\Lambda = \ell *\mu$, consider the following graph on 4 vertices:
\begin{figure}[H]
\vspace{-3.5mm}
\centering\includegraphics[width=.3\textwidth]{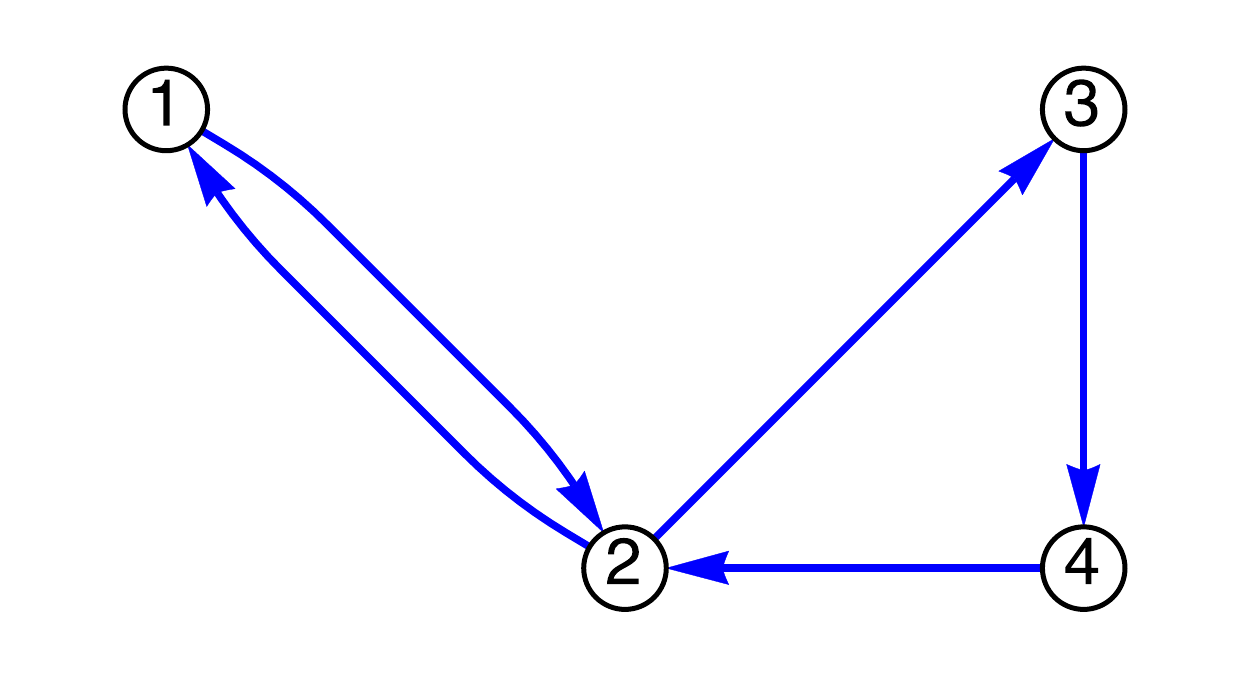}
\vspace{-3mm}
\end{figure} 

\noindent Let $p_1$ be the backtrack and $p_2$ the triangle and let us calculate $\Lambda(p_1p_2)$ and $\Lambda(p_2p_1)$ from $\ell*\mu$. Since the left divisors of $p_1p_2$ are $1,\, p_1$ and $p_1p_2$, we have
\begin{align*}
\Lambda(p_1p_2)&=\ell(1)\mu(p_1p_2)+\ell(p_1)\mu(p_2)+\ell(p_1p_2)\mu(1),\\
&=0\times 0 + 2\times(-1)+5\times 1=3.
\end{align*}
We proceed similarly for $\Lambda(p_2p_1)$:
\begin{align*}
\Lambda(p_2p_1)&=\ell(1)\mu(p_2p_1)+\ell(p_2)\mu(p_1)+\ell(p_2p_1)\mu(1),\\
&=0\times 0 + 3\times(-1)+5\times 1=2.
\end{align*}

\vspace{-2mm}
\noindent Let us now compare these results with a direct calculation of $\Lambda$, by way of counting all the contiguous sequences in the equivalence classes $p_1p_2$ and $p_2p_1$. We find
\begin{align*}
\vspace{-1mm}
w_{21}w_{12}w_{23}w_{34}w_{42}\simeq p_1 p_2,&\qquad w_{12}w_{23}w_{34}w_{42}w_{21}\simeq p_2 p_1,\\
w_{42}w_{21}w_{12}w_{23}w_{34}\simeq p_1 p_2,&\qquad w_{23}w_{34}w_{42}w_{21}w_{12}\simeq p_2 p_1.\\
w_{34}w_{42}w_{21}w_{12}w_{23}\simeq p_1 p_2,&
\end{align*}
This confirms that $\Lambda(p_1p_2) = 3$ and $\Lambda(p_2p_1)=2$, as expected. 
\vspace{2mm}

In number theory, von Mangoldt's explicit formula relates the von Mangoldt Dirichlet series to the zeroes of the zeta function
\begin{equation}\label{RVM}
\psi(x) = \frac 1 {2i\pi} \int_0^\infty \Big( \sum_{n}\frac{\Lambda(n)}{n^s} \Big) \frac{x^s}s ds = x -\!\! \sum_{\rho: \zeta(\rho)=0} \frac{x^\rho}{\rho} - \frac{\zeta'(0)}{\zeta(0)} - \frac{1}{2}\log(1-x^{-2}),
\end{equation}
 for $x>0$ not a prime power. In this expression $\psi(x)=\sum_{p:\,p^k<x} \log(p)$ is the Chebyshev function. Within the framework presented here, Eq.~\eqref{Lambda_trace} and \eqref{zprimeoverz} gives rise to a generalization of this formula, which reduces to counting the closed walks on $G$ from the spectrum of its ordinary adjacency matrix $\mathsf{A}$. 

To see this, observe that in the spirit of the number-theoretic approach, we can look for the extension of the explicit formula Eq.~(\ref{RVM}) to hikes by seeking and then relating two expressions of the log-derivative of $\zeta$: i) a form stemming from the Weierstrass factorisation of zeta; and ii) a form involving the primes explicitly. 
The first form is easy to obtain: on finite graphs, the unlabelled zeta function, being $\zeta_{\mathsf{A}}(s) = 1/\det(\mathsf{I}-e^{-s}\mathsf{A})$, factorises over the spectrum $\text{Sp}(G)$ of the graph. Thus, its log-derivative is $N-\sum_{\rho\in\text{Sp}(G)} (1-e^{-s}\rho)$, as expected from Proposition~\ref{LogDerivative}.

The form of the log-derivative of $\zeta$ involving the primes is more difficult to obtain. In the case of the integers, total commutativity implies that it emerges from the (relatively simple) Euler product. The situation is much more complicated on arbitrary graphs, where the logarithm of the zeta function can be shown to be a branched continued fraction over the primes, known as path-sum \cite{giscard2012continued}. While the general expression of a path-sum on any digraph is available, it is cumbersome and requires knowing all the primes individually along with their commutativity relations. For instance, applying Equation (15b) in \cite{giscard2012continued} to the bi-directed cycle graph $C_5$ on five vertices shown on Figure~\ref{C5graph}, all walks from $v_i$ to itself (excluding the trivial walk $1$) are generated by the continued fraction:
\begin{equation*} \frac{1}{1-\frac{a}{1-\frac{c}{1-\frac{e}{1-b}}}-\frac{d}{1-\frac{b}{1-\frac{e}{1-c}}}-\frac{f}{(1-c) \left(1-\frac{e}{1-c}\right)
   \left(1-\frac{b}{1-\frac{e}{1-c}}\right)}-\frac{g}{(1-b) \left(1-\frac{e}{1-b}\right) \left(1-\frac{c}{1-\frac{e}{1-b}}\right)}} -1 .
\end{equation*}
Noting that $C_5$ is vertex-transitive, the series associated with the unlabelled von Mangoldt function then follows as
$$ \mathcal{S}\Lambda_{\mathsf{A}}(s)=\frac{5}{1-\frac{2 e^{-2s}}{1-\frac{e^{-2s}}{1-\frac{e^{-2s}}{1-e^{-2s}}}}-\frac{2 e^{-5s}}{\left(1-e^{-2s}\right)
   \left(1-\frac{e^{-2s}}{1-e^{-2s}}\right) \left(1-\frac{e^{-2s}}{1-\frac{e^{-2s}}{1-e^{-2s}}}\right)}}-5.
$$
Now, since the graph eigenvalues are $2$, $-\frac{1}{2} \left(1+\sqrt{5}\right)$ (twice) and $\frac{1}{2}
   \left(\sqrt{5}-1\right)$ (twice), the equation relating the primes to the spectrum of $G$ stems from the following equality
\begin{align*}
&\hspace{-5mm}\frac{2}{1+\frac{1}{2} \left(1+\sqrt{5}\right) e^{-s}}+\frac{2}{1+\frac{1}{2}
   \left(1-\sqrt{5}\right) e^{-s}}+\frac{1}{1-2 e^{-s}}~=\\
   &\hspace{30mm}\frac{5}{1-\frac{2 e^{-2s}}{1-\frac{e^{-2s}}{1-\frac{e^{-2s}}{1-e^{-2s}}}}-\frac{2 e^{-5s}}{\left(1-e^{-2s}\right)
   \left(1-\frac{e^{-2s}}{1-e^{-2s}}\right) \left(1-\frac{e^{-2s}}{1-\frac{e^{-2s}}{1-e^{-2s}}}\right)}}.
\end{align*}
%
%

\noindent Extracting the structure of the primes from the unlabelled von Mongoldt series boils down to reconstruct the path-sum from the spectrum alone. 
Clearly, this problem is ill-posed: for a given set of eigenvalues, several path-sums may be possible, in which case they describe non-isomorphic graphs with the same non-zero spectrum. This unfortunate defect provides a means of generating cospectral pairs of directed graphs, which we briefly discuss in \S\ref{sec:constr}. 
For graphs that are determined by their spectrum, the path-sum expression should be fully recoverable but we do not know of a systematic procedure to do so. Thus, in all cases, all that can be stated is that the log-derivative of zeta counts the closed walks of the graph and that the explicit formula reduces to performing this count via the spectrum.


\subsection{Totally additive functions on hikes}\label{TotAddSec} The von Mangoldt identity is a particular case of a more general result concerning totally additive functions over hikes. A function $f:\mathcal H \to \mathbb R$ is said to be totally additive if
$$ \forall h,h' \in \mathcal H \ , \ f(h.h') = f(h) + f(h').$$

\begin{proposition}\label{viennot_additive} Let $f$ be a totally additive function over hikes, then 
$$ f*\mu (h) = \left\{ \begin{array}{cl} f(c) & \text{if $h$ is a non-trivial walk, with $c$ its unique prime right-divisor,} \\ 0 & \text{otherwise.} \end{array} \right.  $$ 
\end{proposition}

\begin{proof} The case $h=1$ follows from noticing that $f$ totally additive implies  $f(1)=0$. Otherwise, write for $h \in \mathcal H \setminus \{1\}$
$$ f*\mu(h) = \sum_{d | h} f(d) \mu \Big( \frac h d \Big).$$
Since $\mu(h/d)$ is zero whenever $h/d$ is not self-avoiding, we may restrict the sum to the divisors $d$ such that $h/d$ is self-avoiding. So, let $c_1,...,c_k$ denote the prime right-divisors of $h$, $s=c_1...c_k$ their product and write $h=h' s$. We get
$$  f*\mu(h) = \sum_{d | h} f(d) \mu \Big( \frac h d \Big) = \sum_{d | s} f(h'd) \mu \Big( \frac s d \Big). $$
The total additivity of $f$ yields 
$$ f*\mu(h) = f(h')  \sum_{d | s}  \mu \Big( \frac s d \Big) +  \sum_{d | s} f(d) \mu \Big( \frac s d \Big)   = \sum_{d | s} f(d) \mu \Big( \frac s d \Big). $$
Now take  the first prime factor $c_1$ of $s$. A divisor of $s$ is either a divisor of $s/c_1$ or $c_1$ times a divisor of $s/c_1$. Thus,
\begin{align*}  \sum_{d | s} f(d) \mu \Big( \frac s d \Big) = & \sum_{d | \frac s{c_1}} \Big[ f(d) \mu \Big( \frac s d \Big) + f(c_1d) \mu \Big( \frac s {c_1d} \Big)  \Big] \\
=& \sum_{d | \frac s{c_1}} \Big[ - f(d) \mu \Big( \frac s {c_1d} \Big) + \big( f(c_1) + f(d) \big) \mu \Big( \frac s {c_1 d} \Big)  \Big] \\
= &  f(c_1)\sum_{d | \frac s{c_1}} \mu \Big( \frac s {c_1 d} \Big) = f(c_1)  \delta \Big( \frac s {c_1} \Big). 
\end{align*}
The final term equals $f(c_1)$ if $s/c_1=1$ (i.e. if $s=c_1$) and zero otherwise. \qquad 
\end{proof}\\

The von Mangoldt identity can be obtained as an application of Proposition \ref{viennot_additive} to the length function $\ell: \mathcal H \to \mathbb N$, which is obviously totally additive. Another notable application concerns a relation between the prime factors counting function $\Omega$ and the indicator function over walks.\\

\begin{corollary}\label{viennot_additive} For all $h \in \mathcal H$, $\Omega*\mu(h) = 1_w(h)$ where $1_w$ is the indicator function over non-trivial walks.\\
\end{corollary}

The proof is omitted as it is a direct application of Proposition \ref{viennot_additive} to the totally additive function $\Omega$. Alternatively, it follows from the fact that the number of prime factors of a hike $h$ (counted with multiplicity) equals the numbers of non-trivial walks dividing it, i.e. $\Omega = 1_w * 1$.\\

\noindent Observe that a similar relation holds in number theory, namely that for $n \in \mathbb N$,
$$\Omega_\mathbb N* \mu_\mathbb N(n) = \sum_{d | n } \Omega_\mathbb N(d) \mu_\mathbb N \Big( \frac n d \Big) = \left\{ \begin{array}{cl} 1 & \text{ if $n=p^k$ with $p$ prime and $k \in \mathbb N$,} \\
0 & \text{ otherwise.} \end{array} \right. $$
Here, $\Omega_\mathbb N(n)$ equals the number of prime factors (counted with multiplicity) of $n \in \mathbb N$ and $\mu_\mathbb N$ is the number theoretic M\"{o}bius function. This parallel indicates that powers of primes are in fact the ``walks'' of number theory, being the only numbers with a unique prime divisor, ergo a unique prime right-divisor. 
 
\subsection{Totally multiplicative functions on hikes}\label{TotMultSec} A consequence of the M\"{o}bius inversion between $\Lambda$ and $\ell$, $\Lambda=\ell*\mu$, concerns totally multiplicative functions on  hikes $f\in\mathcal{F}$. We say that $f$ is totally multiplicative if 
$$ \forall h, h'\in \H \ , \ f(h.h')=f(h)f(h').$$

\begin{lemma}\label{MultFMu} Let $f$ be a totally multiplicative function. The inverse of $f$ through the Dirichlet convolution is given by 
\begin{equation}\label{MultF} f^{-1} = \mu f: h \mapsto \mu(h) f(h) \ , \ h \in \mathcal H.  
\end{equation}
\end{lemma}

\begin{proof}
By a  direct calculation:
\begin{align*}
 (\mu f) * f (h) &=\sum_{d | h} \mu (d) f(d)  f \Big( \frac hd \Big) = f(h) \sum_{d | h} \mu (d) = \delta(h) 
\end{align*}
ending the proof. \end{proof}\\

Lemma~\ref{MultFMu} is reminiscent of the inverse of totally multiplicative functions in number theory, $f^{-1}(n)=\mu_{\mathbb{N}}(n)f(n)$, $n\geq0$, with $\mu_{\mathbb{N}}$ the number-theoretic M\"{o}bius function. The relation between these two results is explained in Section~\ref{NumberTheory}.\\

\begin{corollary}\label{TotMult} 
Let $f$ be a totally multiplicative function on hikes, $F(s)=\mathcal{S}f(s)$  and $F'(s)=dF(s)/ds$. Then,
$$
\frac{F'(s)}{F(s)}=-\sum_{h\in\H}e^{-s\ell(h)} \Lambda(h) f(h) h.
$$
\end{corollary}
 
\begin{proof} For any function $f$, deriving the formal series $F = \mathcal S f$ produces the formal series of $- \ell f$ in view of
$$  \frac d {ds} \sum_{h\in\H}e^{-s \ell(h)}f(h) h=  - \sum_{h\in\H}e^{-s \ell(h)}\ell(h) f(h) h . $$
Furthermore, if $f$ is totally multiplicative, Lemma \ref{MultFMu} gives $1/ F(s) = \sum_{h \in \mathcal H} e^{-s \ell(h)} \mu(h) f(h) h$. We obtain
\begin{align*}
\frac{F'(s)}{F(s)} &=-\sum_{h\in\H}e^{-s \ell(h)} \ell(h) f(h) h \  . \sum_{h\in\H}e^{-s \ell(h)} \mu (h) f (h) h \\
 &=-\sum_{h\in\H}e^{-s \ell(h)} f(h) (\ell*\mu)(h) h \\
&= -\sum_{h\in\H}e^{-s \ell(h)}f(h)  \Lambda(h) h
\end{align*}
where the last equality follows from Proposition \ref{LogDerivative}.\qquad 
\end{proof}\\

An important extension to MacMahon's master theorem stems from the formal series version of Lemma~\ref{MultFMu}:
$$  \mathcal S f(s)=  \sum_{h\in \H} e^{-s \ell(h)} f(h) h = \frac 1 {\sum_{h\in \mathcal H} e^{-s \ell(h)} \mu(h) f(h) h} . $$
To see this, consider first a weighted version of the graph $G$ where all arcs pointing to a vertex $i$ are given a formal weight $t_i$. The weighted adjacency matrix of this weighted graph is $\mathsf{T}\mathsf{W}$, with $\mathsf{T}$ the diagonal matrix where $\mathsf{T}_{ii}=t_i$.
Now observe that a totally multiplicative function on hikes is completely determined by its value on the primes (since $f(hh')=f(h)f(h')$ regardless of the commutativity of $h$ and $h'$). We may therefore consider the totally multiplicative function which associates any prime $p$ with its weight,  
\begin{equation}\label{FMacMahon}
f(p)= \mathrm{weight}(p)=t_{i_2}\cdots t_{i_{\ell(p)}}t_{i_1}.
\end{equation} 
where $\{i_1,\cdots, i_{\ell(p)}\}$ is the set of vertices visited by $p$. Then, Lemma~\ref{MultFMu} yields 
\begin{align} \mathcal{S}f(0)  = \sum_{h\in \H} f(h) h = \frac 1 {\sum_{h\in \mathcal H} \mu(h) f(h) h } 
= \frac 1 {\det(\mathsf{I}- \mathsf{T} \mathsf{W})}.\label{MacMahon} 
\end{align}
This is the non-commutative generalization of MacMahon's theorem discovered by Cartier and Foata \cite{cartier1969}. MacMahon's original result \cite{MacMahon1915} is then recovered upon replacing $\mathsf W$ by the adjacency matrix $\mathsf A$, thus attributing the value $1$ to every hike.\\

In general, totally multiplicative functions on hikes do not have to take on the extremely restricted form of Eq.~(\ref{FMacMahon}). In these cases Lemma~\ref{MultFMu} goes beyond even the non-commutative generalization of MacMahon's theorem. We now present an explicit example illustrating this observation where Lemma~\ref{MultFMu} yields the permanent in relation with a simple totally multiplicative function. \\

In number theory, the Liouville function is defined as $\lambda(n)=(-1)^{\Omega(n)}$, where we recall $\Omega(n)$ is the number of prime factors of the positive integer $n$, counted with multiplicity. We define the walk Liouville function similarly.

\begin{definition}\label{wlambda}
The walk Liouville function $\lambda(h):~\H\to\{-1,1\}$ is defined by $\lambda(h):=(-1)^{\Omega(h)}$, where $\Omega(h)$ is the number of prime factors of $h$, counted with multiplicity.\\
\end{definition}

\noindent The series $\mathcal{S}\lambda(s):=\sum_{h\in\H}e^{-s\ell(h)}\lambda(h) h$ associated to the walk Liouville function has a remarkably simple expression showing that calculating it is $\#$P-complete on arbitrary graphs.
\begin{proposition}\label{LiouvillePerm}
The formal series of the walk Liouville function $\lambda$ satisfies
$$
\mathcal{S}\lambda(s) =\sum_{h\in\H}e^{-s\ell(h)} (-1)^{\Omega(h)} h =  \frac{1}{\perm(\mathsf{I}+e^{-s} \mathsf{W})},
$$
where $\mathrm{perm}$ designates the permanent.\\
\end{proposition}

\begin{proof}
Observe that since $\Omega(h)$ is totally additive, the walk Liouville function is totally multiplicative. A direct application of Lemma \ref{MultFMu} then gives
$$  \lambda^{-1}(h)  = \mu(h) \lambda(h) = \left\{ \begin{array}{cl} 1 & \text{ if $h$ is self-avoiding,} \\ 0 & \text{ otherwise.} \end{array} \right.   $$
The result follows from noticing that $\perm(\mathsf{I}+e^{-s} \mathsf{W}) = \underset{h \text{ self-avoiding}}{\sum} e^{-s \ell(h)} h$. \qquad
\end{proof}\\

The permanent of $\id+e^{-s}\mathsf{W}$ is the series associated to the indicator function on self-avoiding hikes, that is, the absolute value of $\mu$:
$$\perm(\id+e^{-s}\mathsf{W})=\sum_{h\in\H}e^{-s\ell(h)}|\mu(h)| h= \mathcal{S}|\mu|(s). $$
Hence, the walk Liouville function $\lambda$ is the inverse of $|\mu|$ through the Dirichlet convolution, similarly as its number theoretic counterpart.

\subsection{Relation to number theory}\label{NumberTheory} The unique factorization of hikes into products of hikes satisfying the prime property is reminiscent of the fundamental theorem of arithmetic. The  difference between these two results stems from the non-commutativity of the product operation between hikes. Unsurprisingly then, on a graph where all prime cycles commute, the prime factorization of hikes identifies with that of the integers and the poset $P_{G}$ becomes isomorphic to the poset of integers ordered by divisibility, which we denote $P_{\mathbb{N}}$.\\


\begin{theorem}\label{NumberTheoryOnGN} Let $G$ be an infinite directed graph formed by a countable union of vertex-disjoint oriented cycles. Then, the hike poset  $P_{G}$ is isomorphic to the poset $P_{\mathbb{N}}$ of natural integers ordered by divisibility. In particular, the reduced incidence algebra of $P_{G}$, $(\mathcal{F},\ast)$, is isomorphic to the algebra of Dirichlet series equipped with ordinary multiplication. \\
\end{theorem}

\begin{proof}
Let $c_1,c_2,...$ be an enumeration of the simple cycles of $G$. Define $\varphi:\H \longrightarrow \mathbb{N}$ as the function mapping $c_j $ to the $j$-th prime number $p_j$ for all $j \in \mathbb N$ and such that
$$ \forall h,h' \in \mathcal H \ , \ \varphi(h.h') = \varphi(h) \times \varphi(h').  $$
Since $(\mathcal H,.)$ is Abelian, $\varphi$ is one-to-one and more importantly, a  monoid homomorphism from $(\H,.) $ to $(\mathbb{N},\times)$. Consequently, $P_G$ is isomorphic to $ P_{\mathbb{N}}$ and their reduced incidence  algebras are isomorphic as well. Finally, the  reduced incidence algebra of $P_{\mathbb{N}}$ is known to be isomorphic to the algebra of Dirichlet series, see e.g. Example 4.8 on page 282 in \cite{doubilet1972}.\qquad 
\end{proof}\\

In this context, \emph{all results} obtained in Section~\ref{AlgHike} yield valid number theoretic results when applied to the infinite digraph $G$ composed of a countable union of vertex-disjoint simple cycles. For example, on such digraph $G$:

\begin{itemize}
	\item Two simple cycles are vertex-disjoint if, and only if, they are different. Thus, the M\"obius function on $\mathcal H$ coincides with the number theoretic M\"obius function. More generally, the notion of co-primality which extends over hikes to vertex-disjointness is here equivalent to the usual definition of co-primality over $\mathbb N$. 
	\item All prime factors of a hike are also divisors. It follows that $\tau(h) := 1*1(h) $ and $\omega(h):= 1_p*1(h)$ (where  $1_p$ is the prime indicator function) give respectively the number of divisors and the number of prime factors of $h$, similarly as for the arithmetic versions of these functions.
	\item Closed walks take the form $h=p^k$,  for $p$ a prime hike (i.e. a simple cycle) and $k \in \mathbb N$. Thus, Definition \ref{wLambda} yields the von Mangoldt function
$$
\Lambda(h) = \begin{cases}
\ell(p),&\text{if }h=p^k,~p\text{ prime}\\
0,&\text{otherwise}.\end{cases}
$$
This recovers the number-theoretic von Mangoldt function, provided we identify the length of a hike with the logarithm of an integer.
\item More generally, Proposition \ref{viennot_additive} shows that, for any totally additive function $f$, $f*\mu$ has its support on non-trivial walks. On a digraph $G$ where all primes commute, the walks are the powers of primes thus recovering the number theoretic version of the result: for all $f: \mathcal H \to \mathbb R$ totally additive,
$$ f*\mu(h) = \sum_{d | h} f(d) \mu \Big( \frac h d \Big) = \left\{ \begin{array}{cl} f(p) & \text{ if } h = p^k,~p \text{ prime} \\ 0 & \text{ otherwise.} \end{array} \right. $$
\end{itemize}

\section{Relation to the Ihara zeta and the characterisation of graphs}\label{IharaZeta} The Ihara zeta function plays an important role in algebraic graph theory and network analysis as it was shown to relate to \textit{some} properties of the graph \cite{Terras2011}. In this section, we elucidate the relation between the zeta function of the poset of hikes ordered by divisibility and the Ihara zeta function. We then show that the poset $P_G$ and its zeta function $\zeta(s)$ determine undirected graphs uniquely, up to isomorphism.
\vspace{2mm}

\subsection{Ihara zeta function} The basic objects underlying the Ihara zeta function are certain equivalence classes defined over the closed walks of a graph, called primitive orbits \cite{Terras2011}. We begin by recalling basic definitions pertaining to the primitive orbits.
\vspace{2mm}

Two closed walks are said to be equivalent if one can be obtained from the other upon changing its starting point and deleting its immediate backtracks, e.g. $w_{12}w_{23}w_{34}w_{43}w_{31}\simeq w_{23}w_{31}w_{12}$. The resulting equivalence classes on the set of all walks are called \textit{backtrackless orbits}.\footnote{Backtrackless orbits are necessarily connected and may still have one or more backtracks as prime factors.} An orbit is \textit{primitive} if and only if it is not a perfect power of another orbit, i.e. $p_o \not\simeq p_o'^k$, $k>1$. 
The Ihara zeta function is then defined in analogy with the Euler product form of the Riemann zeta function as 
\begin{equation*}
\zeta_I(u):=\prod_{\tilde{p}_o\in \tilde{C}_G}\frac{1}{1-u^{\ell(\tilde{p}_o)}},
\end{equation*}
where $\tilde{C}_G$ is the set of backtrackless primitive orbits on $G$. 
In the following it will be convenient to consider orbits for which immediate backtracks have been retained. In this case, two walks represent the same orbit if and only if one can be obtained from the other upon changing its starting point. In this situation $w_{12}w_{23}\underline{w_{34}w_{43}}w_{31}$ and $w_{12}w_{23}w_{31}$ define different (primitive) orbits. We denote by $C_G$ the set of primitive orbits including those with immediate backtracks.
\vspace{2mm}

Primitive orbits \textit{do not obey the prime property}, that is a primitive orbit may be a factor of the product of two walks $w.w'$ without being a factor of $w$ or $w'$ and the factorization of walks into products of primitive orbits is not unique. We further note that counting primitive orbits is indeed much easier than counting prime cycles.\footnote{In contrast, just determining the existence of a prime cycle of length $n$ on a graph on $n$ vertices is known to be NP-complete, being the Hamiltonian cycle problem.}\\
\begin{proposition}\label{PifromTr}
Let $G$ be a graph and let $\pi_{C_G}(\ell)$ be the number of primitive orbits of length $\ell$ on $G$ with immediate backtracks retained. Then
\begin{equation}\label{ExactPi}
\pi_{C_G}(\ell)=\frac{1}{\ell}\sum_{n|\ell}\mu_{\mathbb{N}}(\ell/n)\,\mathrm{Tr}(\mathsf{A}^n),
\end{equation}
where $\mu_{\mathbb{N}}$ is the number theoretic M\"{o}bius function and $\mathsf{A}$ is the adjacency matrix of $G$.\\
\end{proposition}
\begin{remark} 
\end{remark}
A similar result already exists for backtrackless primitive orbits, in this case $\mathsf{A}$ is replaced by the directed edge-adjacency matrix, see \cite{Terras2011}.\vspace{2mm}

Before we prove Proposition \ref{PifromTr}, it is instructive to relate the zeta function $\zeta(s)$ of $P_G$ to the Ihara zeta function. We start with Eq.~\eqref{log_zeta},
$$\log \zeta(s)= \!\!\sum_{h\neq1} e^{-s\ell(h)}\frac{\Lambda(h)}{\ell(h)}\,h.$$ 
Recall that $\Lambda(h)$ is non-zero only if $h$ is a walk. Furthermore, a walk either defines a primitive orbit or is a power of one, $h= p_o^k$, $k\geq1$, where we write $p_o$ for a primitive orbit in order to avoid confusion with primes. Then we can recast Eq.~\eqref{log_zeta} as
\begin{equation}\label{formallog1}
\log \zeta(s) =  \!\!\sum_{p_o\in C_G}\sum_{k>0} e^{-s\ell(p_o^k)}\frac{\Lambda(p_o^k)}{\ell(p_o^k)}\,p_o^k,
\end{equation}
with $C_G$ the set of primitive orbits on $G$ (including those with immediate backtracks). There are $\ell(p_o)$ walks in the equivalence class $p_o^k$ since two walks are equivalent if and only if one can be obtained from the other upon changing its starting point. Then $\Lambda(p_o^k)=\ell(p_o)$ and Eq.~(\ref{formallog1}) gives
\begin{align}\label{formallog}
\log\zeta(s) &= \sum_{p_o\in C_G}\sum_{k>0} \frac{1}{k}e^{-s k\ell(p_o)} p_o^k,
\end{align}
Exponentiating the series above necessitates some precautions:  being hikes, primitive orbits do not commute $p_o p'_o\neq p'_o p_o$ as soon as $p_o$ and $p'_o$ share at least one vertex. We will present the result of this exponentiation in a future work as it is sufficient for the purpose of relating $\zeta$ with $\zeta_I$ to bypass this difficulty by eliminating all formal variables. This is equivalent to substituting $\mathsf{W}$ with $\mathsf{A}$ in Eq.~(\ref{formallog}). 
This procedure immediately yields
\begin{equation*}
\zeta_{\mathsf{A}}(s):=\frac{1}{\det(\mathsf{I}-e^{-s}\mathsf{A})}=\prod_{p_o\in C_G}\frac{1}{1-e^{-s\ell(p_o)}},
\end{equation*}
this being an Abelianization of $\zeta(s)$. 
Defining $u:=e^{-s}$ now gives
\begin{equation}\label{zetaprimitive}
\zeta_{\mathsf{A}}(u)=\frac{1}{\det(\mathsf{I}-u\mathsf{A})}=\prod_{p_o\in C_G}\frac{1}{1-u^{\ell(p_o)}}.
\end{equation}
We separate the product above into a product over primitive orbits with no immediate backtracks, yielding the Ihara zeta function, and the product $\zeta_b(u):=\prod_{p_b\in C_G}(1-u^{\ell(p_b)})^{-1}$, involving primitive orbits $p_b$ with at least one immediate backtrack. This yields
\begin{equation*}
\zeta_{\mathsf{A}}(u)=\zeta_I(u)\zeta_b(u),
\end{equation*}
which indicates that the Ihara zeta function originates from the unlabeled, Abelianized version $\zeta_{\mathsf{A}}$ of the zeta function $\zeta(s)$. Thus, we may expect $\zeta(s)$ or $P_G$ to hold more information on the graph $G$ than the Ihara zeta function does. In the next section we prove that this is indeed the case as $P_G$ determines undirected graphs uniquely.\\

\begin{proof}[Proof of Proposition~\ref{PifromTr}]
Starting from Eq.~(\ref{zetaprimitive}) we have
$
\det(\mathsf{I}-u\mathsf{A})=\prod_{j=1}^{\infty}(1-u^j)^{\pi_{C_G}(j)}.
$
Taking the logarithm on both sides yields
$$
\sum_{i=1}^\infty\frac{1}{i}\,u^i\, \mathrm{Tr}\, \mathsf{A}^i=\sum_{j=1}^\infty \pi_{C_G}(j)\sum_{k=1}^\infty\frac{u^{kj}}{k},
$$
and the result follows upon equating the coefficients of $z^\ell$ on both sides. Remark that the  product expansion of $\zeta_{\mathsf{A}}$ over the primitive orbits also yields the following expansion for the ordinary resolvent $\mathsf{R}(u)=(u\mathsf{I}-\mathsf{A})^{-1}$ of $G$,
\begin{equation}\label{ResolventLambert}
u^{-1}\,\mathrm{Tr}\,\mathsf{R}(u^{-1})=N+\sum_{\ell\geq 1}\ell\,\pi_{C_G}(\ell)\frac{u^{\ell}}{1-u^\ell},
 \end{equation}
where $\pi_{C_G}(\ell)$ is the number of primitive orbits of length $\ell$ on $G$. We recognize the Lambert series of general term $\ell\,\pi_{C_G}(\ell), \ell \geq 1$. This follows from Eq.~\eqref{log_zeta} together with Eq.~\eqref{zetaprimitive} for $\zeta_\mathsf{A}(u)$.\qquad
\end{proof} 
\vspace{2mm}

\subsection{Characterization of undirected graphs by the hike poset}\label{PGdetermines} We recall that $P_G = (\mathcal H,\prec)$ denotes the poset of hikes $\mathcal H$ partially ordered by left division: $h \prec h' \Longleftrightarrow h | h'$. Observe that the minimal elements of $P_G$ are the simple cycles and two simple cycles $c,c'$ commute if and only if there exists $h$ such that $c \prec h$ and $c' \prec h$. Thus, knowing the poset $P_G$ reduces exactly to knowing the simple cycles and their commutativity relations (i.e. the pairs of intersecting cycles). We now aim to characterize the precise information on a digraph $G$ that is contained in its hike poset $P_G$. \\

Let $\mathsf A$ denote the adjacency matrix of $G$. Replacing $\mathsf W$ by $z \mathsf A$, for $z$ a formal variable, in Eq. \eqref{mobius_inv} yields a slight modification of the characteristic polynomial of $G$ by
$$ \det (\id - z \mathsf A) = \sum_{h \in \mathcal H} \mu(h) z^{\ell(h)}.  $$
This expression indicates that knowing the poset $P_G$, along with the length function $\ell: P_G \to \mathbb N$, suffices to recover the spectrum of a digraph $G$ of known size $N$. On the other hand, the information given by the hike structure $(P_G,\ell)$ is not sufficient to reconstruct $G$ exactly. For instance, all acyclic digraphs on $N$ vertices have their hike posets reduced to $\{1\}$ and are therefore indistinguishable from their hike structure. Actually, non-trivial examples of co-spectral digraphs can be obtained from seeking non-isomorphic digraphs with the same hike structure $(P_G,\ell)$, as we discuss in Section \ref{sec:constr}. \\

Things are somewhat different for undirected graphs which are in fact characterized by their hike structure, provided they contain no isolated vertex, as we prove below. Because hikes are defined from directed edges, let us clarify that we define the hike structure of an undirected graph as that of its bi-directed version, for which each undirected edge is replaced by two arcs of opposite directions. Since the connected components of $G$ are apparent in $P_G$, we shall assume that $G$ is connected without loss of generality.\\

\begin{proposition} Let $G$ be a connected bi-directed digraph. Then, the hike structure $(P_G,\ell)$ determines $G$ uniquely up to isomorphism.\\
\end{proposition}

\begin{proof} Clearly, knowing the length of the hikes allows to identity the backtracks as the minimal elements in $P_G$ of length $2$ (or simply the elements of length $2$ if $G$ has no self-loop). Since two backtracks share a vertex in common if and only if they do not commute, the line graph of $G$ can be recovered from the commutativity relations on the backtracks. The result follows by noting that the bi-directed versions of $K_3$ and $K_{1,3}$ (the only connected undirected graphs not determined by their line graphs, see \cite{whitney1932congruent}) produce different hike posets. \qquad 
\end{proof}

\vspace{2mm}

Of course, the length function provides a significant amount of information and the question remains to know whether an undirected graphs $G$ with no isolated vertex can be reconstructed from the poset $P_G$ only. This turns out to be the case for all undirected graphs but two exceptions: $K_3$, the complete graph on $3$ vertices and $K_{1,5}$, the complete bipartite graph on $1$ and $5$ vertices. \\

\begin{theorem}\label{poset_charac} Let $G$ be a connected bi-directed digraph with no self-loop. The hike poset $P_G$ determines $G$ uniquely up to isomorphism, unless $G \in \{ K_3,K_{1,5} \}$. \\
\end{theorem}

\begin{remark}
\end{remark}
To be able to recover $G$ from its hike poset $P_G$, the additional condition that the graph contains no self-loop is crucial. Indeed, when the length is unknown, a self-loop is indistinguishable from a pendent edge, i.e. an edge with an endpoint of degree one in $G$.\\

Before the proof, let us introduce some definitions. The hike dependence graph $\gamma$ is defined as the undirected graph whose vertices are the simple cycles of $G$, and with an edge between two different vertices $c,c'$ if, and only if, they share at least one vertex (for simplicity, $\gamma$ is defined with no self-loop). The dependence graph $\gamma$ is the complement of the independence graph of $\mathcal H$, which we recall has the simple cycles as vertices and independence relation $\mathcal I_\mathcal H$ as edge set (see Eq. \eqref{dephike}).
Its main role is to provide a visual representation of the commutativity relations between the simple cycles of $G$, so that the poset $P_G$ and $\gamma$ are essentially the same object.\\

Let $c$ be a simple cycle in $G$, the neighborhood of $c$ in $\gamma$ is the set $N_\gamma(c) := \{ c' \neq c : V(c') \cap V(c) \neq \emptyset\}$. We say that $N_\gamma(c)$ is a \textit{clique} if any two elements of $N_\gamma(c)$ are adjacent in $\gamma$.\\

\noindent We introduce a new equivalence relation between cycles. Two simple cycles $c,c'$ are said to be equivalent if they intersect and have the same neighbors in $\gamma$:
$$ c \simeq c' \ \Longleftrightarrow \ N_\gamma(c) \cup \{c\} = N_\gamma(c') \cup \{c'\}. $$
One verifies easily that $\simeq$ is indeed an equivalence relation on the simple cycles.\\

\begin{proof}[Proof of Theorem \ref{poset_charac}] The proof of the theorem relies on recovering the edges of $G$ by identifying the backtracks in the hike dependence graph $\gamma$. In order to do so, we partition the vertices of $\gamma$ into the equivalence classes $[c] := \{c': c \simeq c' \}$ and seek the backtracks in each class. The only ambiguous case turns out to be for $\gamma=K_5$, the complete graph on $5$ vertices.\\

\begin{lemma}\label{step1} Let $G$ be a connected undirected graph and $\gamma$ its hike dependence graph. We assume $\gamma \neq K_5$. If $N_\gamma(c)$ is a clique, then the equivalence class $[c]$ contains only backtracks.\\
\end{lemma}

\begin{proof} Remark that $N_\gamma(c)$ is a clique if and only if $\{ c \} \cup N_\gamma(c)$ is a clique. Thus, if $N_\gamma(c)$ is a clique, the same can be said about all $N_\gamma(c') , c' \simeq c$. It now suffices to prove that $N_\gamma(c)$ being a clique guarantees that $c$ is a backtrack. By contradiction, if $c$ is a simple cycle of length $\ell(c) \geq 3$, then one can always find two non-intersecting backtracks that intersect $c$, unless $c$ is one orientation of an isolated triangle. The latter case is ruled out as $G = K_3$ implies $\gamma=K_5$. \qquad
\end{proof}

\vspace{2mm}

To pursue our goal of identifying the backtracks in $\gamma$, one can now restrict to equivalence classes $C$ for which none of the neighborhoods $N_\gamma(c) , c \in C$ are cliques. \\

\begin{lemma}\label{step2} Let $G$ be a connected undirected graph and $\gamma$ its hike dependence graph. Let $C$ be an equivalence class of size $| C |$ such that, for all $c \in C$, $N_\gamma(c)$ is not a clique. Then,
\begin{itemize}
	\item[i)] If $|C|$ is odd, $C$ contains exactly one backtrack.
	\item[ii)] If $|C|$ is even and $|C| \neq 4$, $C$ contains no backtrack.\\
\end{itemize}
\end{lemma}

\begin{proof} For the proof, we show that there is only one configuration possible for which two backtracks $b,b'$ are equivalent without their neighborhood being cliques, and this configuration contains precisely $4$ elements in its equivalence class. Let $b$ be a backtrack and $i,j$ its endpoints. If $N_\gamma(b)$ is not a clique, there exist two non-intersecting backtracks $b_i,b_j$ that intersect $b$ at $i$ and $j$ respectively, thus forming a path $b_i,b,b_j$. For a backtrack $b'$ to be equivalent to $b$, $b'$ must intersect $b,b_i$ and $b_j$. By symmetry, this reduces to only one possibility represented below.
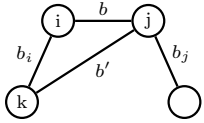
\begin{figure}[h]
\centering
\begin{tikzpicture}[line width=0.3mm,scale=0.4]
\begin{scope}[xshift=0cm]
\vertex[circle,  minimum size=12pt](1) at  (0,0) {i};
\vertex[circle,  minimum size=12pt](2) at  (3,0) {j};
\vertex[circle,  minimum size=12pt](3) at  (4.2,-2.7) {};
\vertex[circle,  minimum size=12pt](4) at  (-1.2,-2.7) {k};

\draw (4)--(1)--(2)--(3);
\draw (4)--(2);
\node[ minimum size=6pt] at  (1.5,0.4) {$b$};
\node[ minimum size=6pt] at  (-1.1,-1.1) {$b_i$};
\node[ minimum size=6pt] at  (4.1,-1.1) {$b_j$};
\node[ minimum size=6pt] at  (1.5,-1.6) {$b'$};
\end{scope}

\end{tikzpicture}
\caption{The only configuration with two equivalent backtracks $b,b'$ whose neighborhoods are not a clique in $\gamma$. Edges may be added from the vertex $j$ but no edge can be added from either $i$ or $k$ without breaking the equivalence of $b$ and $b'$.}
\label{fig:relou}
\end{figure}

\noindent Note that $b$ and $b'$ are no longer equivalent if an edge is added to one endpoint of $b_i$. In this configuration, the only simple cycles intersecting $b_i$ are $b,b'$ and the two orientations $t,t'$ of the triangle formed by $b,b'$ and $b_i$. It follows that $[b] \subseteq \{ b,b',t,t' \}$. On the other hand, all the other neighbors of $b$ in $\gamma$ pass through $j$, a vertex which is common to $b,b',t$ and $t'$. We deduce that $ \{ b,b',t,t' \} \subseteq [b] $, in particular $[b]$ contains $4$ elements.\\
Hence, if we exclude the case $|C|=4$, having no clique among $N_\gamma(c), c \in C$ implies that $C$ contains at most one backtrack. We now come to a simple, although crucial, argument for the proof: two orientations of a simple cycle are equivalent. This means in particular that the number of backtracks in an equivalence class of size $|C|$ has the same parity as $|C|$. This concludes the proof of the lemma.\qquad
\end{proof}

\vspace{2mm}

The situation displayed in Figure \ref{fig:relou} must be tackled separately, leading to the final part of the backtrack identification process.\\

\begin{lemma}\label{step3} Let $G$ be a connected undirected graph and $\gamma$ its hike dependence graph. Let $C$ be an equivalence class of size $| C |=4$ such that, $\forall c \in C$, $N_\gamma(c)$ is not a clique. Denote by $T(C)$ the set of neighbors common to all simple cycles in $C$: 
$$  T(C) = \underset{c \in C}{\bigcap} N_\gamma(c). $$
If there exists $c' \in T(C)$ with exactly $4$ neighbors in $\gamma$, then there are exactly two backtracks in $C$. Otherwise, $C$ contains no backtrack.\\
\end{lemma}

\begin{proof} We start by noting that $T(C)$ contains no element in $C$, in fact $T(C) = N_\gamma(c) \setminus C$ for any $c \in C$. Since $|C|=4$ (and we ruled out the situation of Lemma \ref{step1}), we have either two backtracks in $C$ or none. In the former case, we fall precisely in the configuration displayed in Figure \ref{fig:relou}, the only configuration with two equivalent backtracks whose neighborhoods are not cliques. In this case, using the same notation, $C = \{b,b',t,t'\}$ and $b_i \in T(C)$ has four neighbors in $\gamma$, namely $b,b',t,t'$. On the other hand, if $C$ contains no backtrack, it comprises the two orientations of two simple cycles of length $\geq 3$, which we denote by $c_1,c_1',c_2,c_2'$. Any common neighbors $c'$ of $c_1,c_1',c_2,c_2'$ in $\gamma$ also intersects an edge, and thus can not satisfy the condition $N_\gamma(c') =4$.\qquad
\end{proof}

\vspace{2mm}

Since all the vertices of $\gamma$ belong to one equivalence class, this lemma shows that all backtracks can be identified in the hike dependence graph $\gamma$, unless $\gamma = K_5$, which is not covered by Lemma \ref{step1}.  The line graph of $G$ can then be obtained as the subgraph of $\gamma$ induced by the backtracks. It is well-known that an undirected graph $G$ is uniquely determined by its line graph unless $G \in \{ K_3,K_{1,3} \}$. Thus, these two cases must be dealt with separately. If $G= K_3$, the hike dependence graph is the only problematic case $\gamma = K_5$. On the other hand, $G=K_{1,3}$ is the only undirected graph with hike dependence graph $\gamma=K_3$. Thus, the only case where $G$ is not recoverable from $\gamma$ is for $\gamma = K_5$, which corresponds to $G \in \{ K_3,K_{1,5} \}$. This concludes the proof.\qquad
\end{proof}

\vspace{2mm}

For sake of clarity, let us consider an example where we reconstruct an undirected graph $G$ from its hike dependence graph $\gamma$, following the three main steps of the proof of Theorem \ref{poset_charac}. 

\begin{figure}[h]
\centering
\begin{tikzpicture}[line width =0.3mm,scale=0.6]
\begin{scope}[xshift=0cm]
\vertex[circle,  minimum size=7pt](1) at  (0,1) {};
\vertex[circle,  minimum size=7pt](2) at  (0,-1) {};
\vertex[circle,  minimum size=7pt](3) at  (2,0) {};
\vertex[circle,  minimum size=7pt](4) at  (4,0.9) {};
\vertex[circle,  minimum size=7pt](5) at  (4,-0.9) {};
\vertex[circle,  minimum size=7pt](6) at  (5,2.5) {};
\vertex[circle,  minimum size=7pt](7) at  (5,-2.5) {};
\vertex[circle,  minimum size=7pt](8) at  (7,0) {};

\draw (1)--(2)--(3)--(1);
\draw (3)--(4)--(5)--(3);
\draw (3)--(6)--(7)--(3);
\draw (8)--(6)--(7)--(8);
\draw (8)--(4)--(5)--(8);
\draw (4)--(7)--(5)--(6)--(4);
\end{scope}
\end{tikzpicture}
\caption{A hike dependence graph $\gamma$.}
\label{fig:ex}
\end{figure}
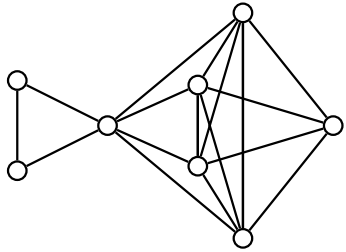

\textbf{Step 1:}\textit{ Partition the vertices into equivalence classes.} The equivalence classes are the maximal sets of adjacent vertices that share the same neighbors other than themselves (the maximal sets that are both a clique and a module). In the example, there are four equivalence classes represented in Figure \ref{fig:ex2}.\\

\textbf{Step 2:} \textit{Identify the backtracks within each equivalence class. } If the union of an  equivalence class and its common neighbors form a clique, all its elements are backtracks (e.g. the class I). Otherwise, equivalence classes of size $k \neq 4$ contain $1$ backtrack if $k$ is odd (e.g. classes II and IV) and no backtrack otherwise. If none of the two previous conditions is verified, an equivalence class contains either two backtracks if one of its neighbors has a neighborhood of size $4$ (e.g. the class III) and no backtrack otherwise.\\

\textbf{Step 3:}\textit{ Extract the line graph of $G$.}  The line graph of $G$ is the subgraph of $\gamma$ induced by the backtracks. 

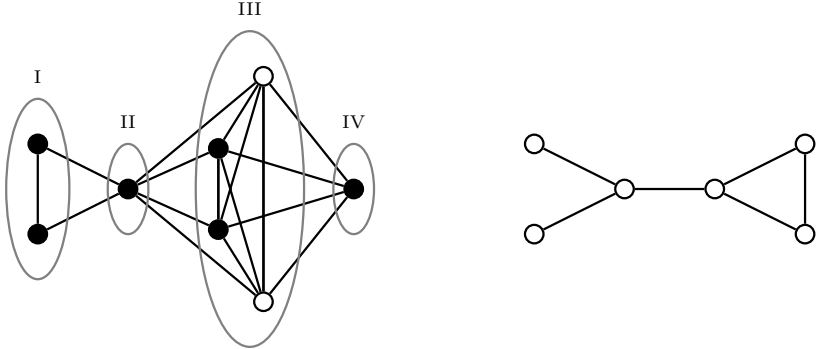
\begin{figure}[H]
\centering
\begin{tikzpicture}[line width =0.3mm,scale=0.6]

\begin{scope}[xshift=0cm]
\vertex[circle,  fill=black, minimum size=7pt](1) at  (0,1) {};
\vertex[circle,  fill=black,  minimum size=7pt](2) at  (0,-1) {};
\vertex[circle,  fill=black,  minimum size=7pt](3) at  (2,0) {};
\vertex[circle,  fill=black,  minimum size=7pt](4) at  (4,0.9) {};
\vertex[circle,  fill=black,  minimum size=7pt](5) at  (4,-0.9) {};
\vertex[circle,  minimum size=7pt](6) at  (5,2.5) {};
\vertex[circle,  minimum size=7pt](7) at  (5,-2.5) {};
\vertex[circle,  fill=black,  minimum size=7pt](8) at  (7,0) {};

\draw (1)--(2)--(3)--(1);
\draw (3)--(4)--(5)--(3);
\draw (3)--(6)--(7)--(3);
\draw (8)--(6)--(7)--(8);
\draw (8)--(4)--(5)--(8);
\draw (4)--(7)--(5)--(6)--(4);

\draw[gray] (0,0) ellipse (0.7cm and 2cm);
\draw[gray] (2,0) ellipse (0.45cm and 1cm);
\draw[gray] (4.7,0) ellipse (1.2cm and 3.5cm);
\draw[gray] (7,0) ellipse (.45cm and 1cm);

\node[minimum size=7pt]() at  (0,2.5) {I};
\node[minimum size=7pt]() at  (2,1.5) {II};
\node[minimum size=7pt]() at  (4.7,4) {III};
\node[minimum size=7pt]() at  (7,1.5) {IV};

\end{scope}

\begin{scope}[xshift=11cm]
\vertex[circle, minimum size=7pt](1) at  (0,1) {};
\vertex[circle,  minimum size=7pt](2) at  (0,-1) {};
\vertex[circle,  minimum size=7pt](3) at  (2,0) {};
\vertex[circle,  minimum size=7pt](4) at  (4,0) {};
\vertex[circle,  minimum size=7pt](5) at  (6,1) {};
\vertex[circle,  minimum size=7pt](6) at  (6,-1) {};

\draw (2)--(3)--(1);
\draw (3)--(4)--(5)--(6)--(4);

\end{scope}

\end{tikzpicture}
\caption{On the left, the hike dependence graph $\gamma$ partitioned into  equivalence classes in order to identify the backtracks and extract the line graph (black vertices). On the right, the undirected graph $G$ is deduced from the line graph.}
\label{fig:ex2}
\end{figure}

\subsection{Construction of co-spectral non-isomorphic digraphs}\label{sec:constr}

It is a basic, rarely questioned, tenet of network analysis that walks accurately reflect the properties of the network on which they take place. Following this tenet most techniques used to distinguish networks are walk-based, see e.g. \cite{Cooper2012, Estrada1, Estrada2} and references therein. Yet, the failure of Theorem~\ref{poset_charac} on directed graphs shows that this tenet is incorrect, at least for closed walks, even if the graphs considered are strongly connected. Consequences of this observation in network analysis as well as methods to generate pairs of non-isomorphic digraphs with identical walk sets are discussed in this section. \\

First, we observe that pairs of cospectral non-isomorphic digraphs can be generated by relying on graph-transformations that leave the underlying hike poset invariant. Indeed, because the poset determines the spectrum, digraphs with the same hike structure are necessarily co-spectral. The existence of such transformations is not trivial and is strongly hinted at by the failure of Theorem~\ref{poset_charac} on directed graphs.

We thus looked for such transformations and identified two of them pertaining to pairs of simple cycles. Indeed, if two simple cycles intersect each-other, and that at least one vertex of their intersection only belongs to these two cycles, then so long as the intersection is maintained, we may add or remove vertices belonging to both cycles without modifying the poset. A symmetric removal or addition of vertices somewhere else in the graph will keep the number of vertices constant, ensuring that no additional zero eigenvalues will appear. We may also simply move the point(s) of intersection of two simple cycles in the graph without modifying the hike structure. Although discovered empirically, the availability of these transformations  can be tested systematically on any given-digraph. If this test is positive, then we can generate at least one connected non-isomorphic co-spectral partner to the original digraph.
Below is an example of partners generated this way:
\begin{figure}[H]
\centering
\begin{tikzpicture}[line width=.3mm,scale=0.6]
\begin{scope}[xshift=0cm]
\vertex[circle,  minimum size=7pt](1) at  (0,0) {};
\vertex[circle,  minimum size=7pt](2) at  (2,1) {};
\vertex[circle,  minimum size=7pt](3) at  (2,-1) {};
\vertex[circle,  minimum size=7pt](4) at  (4,1) {};
\vertex[circle,  minimum size=7pt](5) at  (4,-1) {};
\vertex[circle,  minimum size=7pt](6) at  (6,-1) {};

\draw[->] (1) edge (2);
\draw[->] (2) edge (3);
\draw[->] (3) edge (1);
\draw[->] (3) edge (5);
\draw[->] (5) edge (4);
\draw[->] (4) edge (2);
\draw[->] (5) edge (6);
\draw[->] (6) edge (5);

\end{scope}

\begin{scope}[xshift=9cm]
\vertex[circle,  minimum size=7pt](1) at  (0,-1) {};
\vertex[circle,  minimum size=7pt](7) at  (1,1) {};
\vertex[circle,  minimum size=7pt](2) at  (2,1) {};
\vertex[circle,  minimum size=7pt](3) at  (2,-1) {};
\vertex[circle,  minimum size=7pt](4) at  (4,1) {};
\vertex[circle,  minimum size=7pt](5) at  (4,-1) {};

\draw[<-] (1) edge (7);
\draw[<-] (7) edge (3);
\draw[<-] (2) edge (3);
\draw[<-] (3) edge (1);
\draw[<-] (3) edge (5);
\draw[<-] (4) edge (2);
\draw[<-] (5) edge (4);
\draw[<-] (4) edge (5);

\end{scope}

\end{tikzpicture}
\caption{\label{FigG1G2}Two strongly connected co-spectral digraphs with the same hike structure. The graphs are so similar that all their immantal polynomials ($\det(\mathsf{I} - z\mathsf{A})$, $\perm(\mathsf{I} - z\mathsf{A})$, etc.) are identical.}
\end{figure}
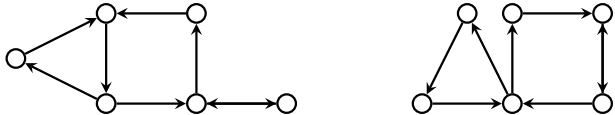  
%
%
%
%
%
%
%
%
%
%
%
%
%
There is little doubt that there exist further, less simple, graph transformations leaving the hike structure invariant, this time involving not two but three or more simple cycles. A systematic effort to find such transformations is therefore interesting.  Now the main consequence in network analysis of the preliminary results presented in this section is that methods based exclusively on cycles will never be able to separate certain pairs of graphs. 

A corollary of this observation concerns the discriminating power of the set of all immanants of a graph. It is already known that some graphs, in particular almost all trees, share some of their immanantal polynomials \cite{Schwenk1973,Merris1991}. We may see the failure of Theorem~\ref{poset_charac} as a generalisation of this result to digraphs: there exists an infinite number of pairs of digraphs with the same hike structure and thus the same set of \emph{all} their immanantal polynomials. In fact, since non-isomorphic trees have, by  Theorem~\ref{poset_charac}, different hike structures, this result appears to be stronger than that known on trees in one respect: it shows that some additional information carried by the hike structure beyond that given by all the immanants is still not sufficient to uniquely characterise some directed graphs.\\

We conclude this section with a last method to generate cospectral pairs of non-isomorphic graphs, this time via partial expansions of the von-Mangoldt function that change the hike structure but preserve the set of non-zero eigenvalues. 
The method relies on the observation that two graphs on $N$ vertices have the same set of non-zero eigenvalues if and only if for all strictly positive $0<\ell\leq N$, they have the same number of closed walks of length $\ell$. 
Since the series associated to the von Mangoldt function is precisely the series of all walks of strictly positive length, two connected digraphs $G$ and $G'$ sharing the same von Mangoldt series  $\mathcal{S}\Lambda(s)$ have the same non-zero spectrum. 
Thus, the path-sum formulation of this series suggests that $G'$ can be generated from $G$ using graph transformations that leave the path-sum invariant. 
A simple example illustrating these observations is shown below:
\vspace{-2mm}
\tikzset{every loop/.style={min distance=5mm,in=0,out=80,looseness=8}}
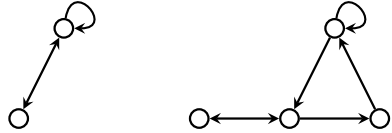
\begin{figure}[H]
\centering
\begin{tikzpicture}[line width=.3mm,scale=0.6]
\begin{scope}[xshift=1cm]
\vertex[circle,  minimum size=7pt](1) at  (0,0) {};
\vertex[circle,  minimum size=7pt](2) at  (1,2) {};
\draw[<->] (1) edge (2);
\draw[->] (2) edge  [loop above,thick,minimum size=10pt] (2);
\end{scope}

\begin{scope}[xshift=7cm]
\vertex[circle,  minimum size=7pt](1) at  (-2,0) {};
\vertex[circle,  minimum size=7pt](2) at  (0,0) {};
\vertex[circle,  minimum size=7pt](3) at  (2,0) {};
\vertex[circle,  minimum size=7pt](4) at  (1,2) {};
\draw[<->] (1) edge (2);
\draw[->] (2) edge (3);
\draw[->] (3) edge (4);
\draw[->] (4) edge (2);
\draw[->] (4) edge  [loop above,thick,minimum size=10pt] (4);
\end{scope}
\end{tikzpicture}
\caption{\label{Ex1PS}Two non-isomorphic directed graphs with the same non-zero spectrum.}
\end{figure}
\vspace{-2mm}  
\noindent The unlabelled von Mangoldt series of graph on the left has path-sum expression
\begin{equation*}
\mathcal{S}\Lambda_{\text{left}}(s)=\frac{1}{1-e^{-s}-e^{-2s}}+\frac{1}{1-\frac{e^{-2s}}{1-e^{-s}}}-2= \frac{e^s+2}{e^{2 s}-e^s-1},
\end{equation*}
while that of the graph on the right is
\begin{align*}
\mathcal{S}\Lambda_{\text{right}}(s)&=\frac{1}{1-e^{-s}-\frac{e^{-3s}}{1-e^{-2s}}}+\frac{1}{1-\frac{e^{-3s}}{(1-e^{-s})
   (1-e^{-2s})}}+\frac{1}{1-\frac{e^{-2s}}{1-\frac{e^{-3s}}{1-e^{-s}}}}+\frac{1}{1-e^{-2s}-\frac{e^{-3s}}{1
   -e^{-s}}}-4,\\
   &= \frac{e^s+2}{e^{2 s}-e^s-1}.
\end{align*}
That is, their unlabelled von Mangoldt series are equal $\mathcal{S}\Lambda_{\text{left}}(s)=\mathcal{S}\Lambda_{\text{right}}(s)$. In particular, it follows that both graphs have the same non-zero eigenvalues $\frac{1}{2} \left(1\pm\sqrt{5}\right)$. 
%
%

The transformation that we used to generate the exemple of Fig.~(\ref{Ex1PS}) is a partial expansion of the continued fraction for $\mathcal{S}\Lambda(s)$ into Taylor series. Indeed, consider a piece of path-sum of the form $1/(1-e^{-s\ell_1}/(1-e^{-s\ell_2}))$. Since e.g. $e^{-s\ell_1}/(1-e^{-s\ell_2}) = e^{-s\ell_1} + e^{-s(\ell_1+\ell_2)}/(1-e^{-s\ell_2})$ we have 
$$
\frac{1}{1-\frac{e^{-s\ell_1}}{1-e^{-s\ell_2}}}=\frac{1}{1-e^{-s\ell_1}-\frac{e^{-s(\ell_1+\ell_2)}}{1-e^{-s\ell_2}}}.
$$
On the graph this equality translates into the following general transformation preserving the von Mangoldt series:
\begin{figure}[H]
\begin{center}
\includegraphics[width=0.4\textwidth]{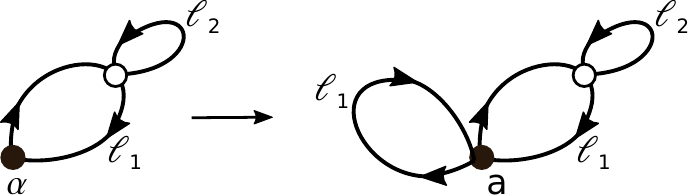}

\caption{\footnotesize{A transformation of directed graphs preserving the non-zero spectrum.}}
\label{fig:ps}
\end{center}
\end{figure}

\noindent This procedure is very general: any piece of a path-sum can be expanded at any order, and the result will always describe the von Mangoldt series of a directed graph with the same non-zero spectrum as the initial one. Finally, several such transformations can be used in conjunction to keep the overall number of vertices constant and thus generate true cospectral pairs.\\

%

\section{Conclusion}

Our results demonstrate that an ``algebraic theory of hikes" can be developed in close parallel to number theory. Although hikes only form a semi-commutative monoid, an equivalent to the fundamental theorem of arithmetic holds on it and recovers a plethora of relations between formal series, with applications in both general combinatorics and number theory. For example, we found that MacMahon's master theorem and the number-theoretic inverse of a totally multiplicative function $f$ over the integers $f(n)^{-1}=f(n)\mu_{\mathbb{N}}(n)$ \cite{Apostol1976}, can be linked to the same general result concerning hikes.
\vspace{2mm}  

We believe that our approach also offers a novel perspective on outstanding open problems of enumerative combinatorics on graphs. Most notably, proving asymptotic estimates for the number 
of self-avoiding paths 
on infinite regular lattices 
corresponds to establishing the prime number theorem for hikes. 
In this respect, an ``algebraic theory of hikes" would find itself in the situation of number theory in the mid 19th-century. Accordingly, partial progress towards asymptotic prime-counting may be possible via a better understanding of the relation between the zeta function of $P_G$ and the primes.
\vspace{2mm}

Some results pertaining to this relation have been left out of the present study because of length considerations and will be presented in future works. In particular, i) there exists an exact relation between $\zeta$ and the ordinary generating function of the primes; ii) $\zeta$ admits an infinite product expansion giving rise to a functional equation on at least some types of graphs; and iii) its logarithm is a branched continued fraction involving only the primes. Furthermore, $\zeta$-based systematic procedures for enumerating certain types of hikes are readily available. Observe indeed that Eq.~(\ref{zprimeoverz}) indicates that the set of hikes with non-zero coefficient in $\log \zeta$, called the support of $\log \zeta$, is the set of closed walks. In fact, the logarithm of $\zeta$ is one of the simplest member of an infinite family of hypergeometric functions of $\zeta$, whose supports are sets of hikes obeying precise  constraints. For example, the support of  $2 -2(\log \zeta+1)/\zeta$ is the set of hikes $h=p_1\cdots p_n$ for which there exists a prime $p_i$ such that $h=p_1\cdots p_{i-1}p_{i+1}\cdots p_n$ is a walk. 
In every case, exploiting the spectral decomposition of $\zeta$ provides explicit formulas for counting the hikes of the support from the spectrum of $\mathsf{A}$.

\section*{Acknowledgements}
P.-L. Giscard acknowledges financial support from the Royal Commission for
the Exhibition of 1851. The authors are grateful to Thibault Espinasse and Xavier Viennot for many insightful discussions on the theorey of heaps of pieces, and to two anonymous referees for their thorough revisions and constructive remarks.
%
%
 \bibliographystyle{siam}
 \bibliography{AlgComb_SIAM}

\end{document}